\documentclass{amsart}

\usepackage{amsmath,amsfonts,amssymb,color,hyperref, enumerate}

\usepackage{amsmath, amsfonts,amsthm,amssymb,amscd, verbatim,graphicx,color,multirow,tikz,tikz-cd,booktabs, caption, mathdots,bm,chngcntr,adjustbox,longtable
}
\usepackage{tikz-cd}
\usetikzlibrary{positioning}
\newtheorem{theorem}{Theorem}[section]
\newtheorem{lemma}[theorem]{Lemma}
\newtheorem{proposition}[theorem]{Proposition}

\newtheorem{corollary}[theorem]{Corollary}

\newtheorem{conjecture}[theorem]{Conjecture}

\begin{document}

\title[Cliques in derangement graph]{Cliques in  derangement graphs for innately transitive groups}

\author[M. Fusari]{Marco Fusari}
\address{Marco Fusari, Dipartimento di Matematica ``Felice Casorati", University of Pavia, Via Ferrata 5, 27100 Pavia, Italy}
\email{lucamarcofusari@gmail.com}

\author[A.~Previtali]{Andrea Previtali}
\address{Andrea Previtali, Dipartimento di Matematica e Applicazioni, University of Milano-Bicocca, Via Cozzi 55, 20125 Milano, Italy}
\email{andrea.previtali@unimib.it}

\author[P.~Spiga]{Pablo Spiga}
\address{Pablo Spiga, Dipartimento di Matematica e Applicazioni, University of Milano-Bicocca, Via Cozzi 55, 20125 Milano, Italy}
\email{pablo.spiga@unimib.it}

\begin{abstract}

  Given a permutation group $G$, the derangement graph of $G$ is the Cayley graph with connection set the derangements of $G$. The group $G$ is said to be innately transitive if $G$ has a transitive minimal normal subgroup. Clearly, every primitive group is innately transitive. We show that, besides an infinite family of explicit  exceptions, there exists a function $f:\mathbb{N}\to \mathbb{N}$ such that, if $G$ is innately transitive of degree $n$ and the derangement graph of $G$ has no clique of size $k$, then $n\le f(k)$.

   Motivation for this work arises from investigations on Erd\H{o}s-Ko-Rado type theorems for permutation groups.

\keywords{derangement graph, cliques, independent sets, Erd\H{o}s-Ko-Rado Theorem, primitive groups, multipartite graphs}
\end{abstract}

\subjclass[2010]{Primary 05C35; Secondary 05C69, 20B05}

\maketitle

\section{Introduction}\label{sec:intro}
One of the most beautiful results in extremal combinatorics is the Erd\H{o}s-Ko-Rado theorem~\cite{erdos1961intersection}:
let $n$ and $k$ be positive integers with $1\le 2k<n$ and let $\mathcal{F}$ be a family of $k$-subsets of $\{1,\ldots,n\}$.
If any two elements from $\mathcal{F}$ intersect in at least one point, then $|\mathcal{F}|\le \binom{n-1}{k-1}$. Moreover,
the inequality is attained if and only if there exists $x\in\{1,\ldots,n\}$ such that each element from $\mathcal{F}$ contains $x$.

There are various analogues of the Erd\H{o}s-Ko-Rado theorem for a number of combinatorial structures. In this paper we are interested in the analogue for permutation groups. Let $G$ be a finite permutation group on $\Omega$. A subset $\mathcal{F}$ of $G$ is said to be \textit{\textbf{intersecting}} if, for any two elements $g,h\in \mathcal{F}$, $gh^{-1}$ fixes some point of $\Omega$. This is a very natural definition; indeed, by writing $g$ as the $n$-tuple $(1^g,2^g,\ldots,n^g)$, we see that $gh^{-1}$ fixes some point of $\Omega$ if and only if the $n$-tuples corresponding to $g$ and $h$ agree in at least one coordinate. Therefore, somehow, this mimics the definition of intersecting sets in the original Erd\H{o}s-Ko-Rado theorem.

Observe that, for every $\omega\in \Omega$, the point stabilizer $G_\omega$ is intersecting. More generally, each coset of the stabilizer of a point is an intersecting set. Answering a question of Erd\H{o}s, Cameron-Ku~\cite{6} and Larose-Malvenuto~\cite{10} have independently proved an analogue of the Erd\H{o}s-Ko-Rado theorem when $G=\mathrm{Sym}(\Omega)$.\footnote{An intersecting set of $\mathrm{Sym}(\Omega)$ has cardinality at most $(|\Omega|-1)!$; moreover, the intersecting sets attaining the bound $(|\Omega|-1)!$ are cosets of the stabilizer of a point.}  Unfortunately, in general only rarely $G_\omega$ is an intersecting set of maximal size in $G$\footnote{For instance, if we let the alternating group $\mathrm{Alt}(5)$ acting on the ten $2$-subsets of $\{1,2,3,4,5\}$, we see that $\mathrm{Alt}(4)$ is an intersecting set of size $12$, whereas the point stabilizer in this action has only cardinality $6$.} and hence no analogue of the Erd\H{o}s-Ko-Rado theorem holds for arbitrary permutation groups. Even when $|G_\omega|$ is the maximal cardinality of an intersecting set for $G$, it is far from being true that all intersecting sets attaining the bound $|G_\omega|$ are cosets of the stabilizer of a point.\footnote{For instance, in the projective general linear group $G=\mathrm{PGL}_d(q)$ in its $2$-transitive action on the $(q^d-1)/(q-1)$ points of the projective space $\mathrm{PG}_{d-1}(q)$, the intersecting sets of maximal cardinality are either cosets of the stabilizer of a point or cosets of the stabilizer of a hyperplane, see~\cite{spiga}.} These two difficulties make investigations on intersecting sets of maximal size in arbitrary permutation groups more interesting and challenging.

Let $\omega\in \Omega$ with $G_\omega$ having maximum cardinality among point stabilizers.\footnote{Observe that all point stabilizers have the same cardinality when $G$ is transitive.} The \textit{\textbf{intersection density}} of the intersecting family $\mathcal{F}$ of $G$ is defined by
$$\rho(\mathcal{F})=\frac{|\mathcal{F}|}{|G_\omega|}.$$
The \textit{\textbf{intersection density}} of $G$ is
$$\rho(G)=\max\{\rho(\mathcal{F})\mid \mathcal{F}\subseteq G, \mathcal{F} \hbox{ is intersecting}\}.$$
This invariant was introduced by Li, Song and Pantagi in~\cite{li2020ekr} to measure how ``close'' $G$ is from satisfying the Erd\H{o}s-Ko-Rado theorem.

Let $\mathcal{D}$ be the set of all \textit{\textbf{derangements}} of $G$,
where a derangement is a permutation without fixed points. The derangement graph
of $G$ is the graph $\Gamma_G$ whose vertex set is the set $G$ and whose edge set consists of
all pairs $(h, g) \in G \times G$ such that $gh^{-1} \in \mathcal{D}$. Thus, $\Gamma_G$ is the Cayley graph
of $G$ with connection set $\mathcal{D}$. With this terminology, an intersecting family of $G$ is an \textit{\textbf{independent
set}} or \textit{\textbf{coclique}} of $\Gamma_G$, and vice versa. As customary, we denote by $\omega(\Gamma_G)$ the maximal size of a clique and by $\alpha(\Gamma_G)$ the maximal size of a coclique (a.k.a. independent set).

Now, the clique-coclique bound~\cite[Theorem~2.1.1]{GodsilMeagher}
\begin{equation}\label{clique-coclique}\alpha(\Gamma_G)\omega(\Gamma_G)\le |V\Gamma_G|=|G|
\end{equation}
can be used to extract useful information on the intersection density of $G$. Indeed, from~\eqref{clique-coclique} and from the definition of intersection density, we obtain
\begin{equation}\label{clique-coclique2}\rho(G)\le \frac{|\Omega|}{\omega(\Gamma_G)}.
\end{equation}

When $G$ is transitive and $|\Omega|\ge 2$, Jordan's theorem\footnote{See~\cite{Serre} for a beautiful account of Jordan's theorem and for a number of applications in various areas of mathematics.} ensures that $G$ has a derangement $g$ and hence $\{1,g\}$ is a clique of $\Gamma_G$ of cardinality $2$. Therefore,~\eqref{clique-coclique2} yields $\rho(G)\le |\Omega|/2$.

Theorem~1.5 in~\cite{KRS} shows that, when $G$ is transitive and $|\Omega|\ge 3$, the derangement graph $\Gamma_G$ has a clique of cardinality $3$, that is a triangle, and hence $\rho(G)\le |\Omega|/3$. Despite the fact that Jordan's theorem is elementary, the proof of~\cite[Theorem~1.5]{KRS} is quite involved and ultimately relies on the Classification of the Finite Simple Groups.

In the light of these two results, Question~6.1 in~\cite{KRS} asks for the existence of a function $f:\mathbb{N}\to\mathbb{N}$ such that, if $G$ is transitive of degree $n$ and $\Gamma_G$ has no $k$-clique, then $n\le f(k)$. Indeed, when $k=2$, we have $n\le 1$ by Jordan's theorem and, when $k=3$, we have $n\le 2$ by~\cite[Theorem~1.5]{KRS}. A similar question, formulated in terms of (weak) normal coverings of groups is in~\cite{BSW}.

From~\cite{Bob} and~\cite{Praeger,Praeger1,Praeger2}, we see that there are remarkable applications of normal coverings and Kronecker classes in algebraic number theory. In fact, Question~6.1 in~\cite{KRS} is very much related to conjectures of Neumann and Praeger on coverings of finite groups for applications on Kronecker classes.  We have summarized in Section~\ref{coveringsKronecker} these applications and the connection with our work.

In this paper we make the first substantial progress towards this question. A permutation group $G$ on $\Omega$ is said to be \textit{\textbf{innately transitive}} if $G$ has a minimal normal subgroup $N$ with $N$ transitive on $\Omega$. These permutation groups greatly generalize the class of primitive and quasiprimitive groups. Moreover, innately transitive groups admit a structural result similar to the  O'Nan-Scott   theorem for primitive and quasiprimitive groups~\cite{BP}; furthermore, they play a substantial role in a number of questions in finite permutation groups, see for instance~\cite{Giudici}.
\begin{theorem}\label{thrm:main}
There exists a function $f_1:\mathbb{N}\to\mathbb{N}$ such that, if $G$ is innately transitive of degree $n$ and the derangement graph of $G$ has no clique of size $k$, then $n\le f_1(k)$.
\end{theorem}
In particular, we answer~\cite[Question~6.1]{KRS} when the permutation group $G$ is innately transitive. We make no particular effort in optimizing the function $f$ in Theorem~\ref{thrm:main}, except when $k= 4$.\footnote{We make a special effort in characterizing the innately transitive groups $G$ such that $\Gamma_G$ has no clique of size  $4$, because in the future we intend to use this result to classify arbitrary transitive groups $G$ with $\Gamma_G$ having no clique of size $4$.}
\begin{theorem}\label{thrm:main1}
If $G$ is innately transitive of degree $n$ and the derangement graph of $G$ has no clique of size $4$, then $n\le 3$.
\end{theorem}

A permutation group $X$ on $\Omega$ is said to be \textbf{\textit{semiregular}} if no non-identity element of $X$ fixes some point of $\Omega$, that is, $X_\omega=1$ $\forall\omega\in \Omega$. Observe that a semiregular subgroup $X$ forms a clique in the derangement graph of $\mathrm{Sym}(\Omega)$.\footnote{Indeed, the matrix having rows indexed by the elements of $X$, columns indexed by the elements of $\Omega$ and having $i^g$ in row $g$ and column $i$ is a partial Latin square and hence a clique in the derangement graph.}  In fact, Theorems~\ref{thrm:main} and~\ref{thrm:main1} both follow from a more general result concerning semiregular subgroups in innately transitive permutation groups.

\begin{theorem}\label{thrm:main2}
There exists a function $f_2:\mathbb{N}\to\mathbb{N}$ such that, if $G$ is innately transitive of degree $n$ and $G$ has no semiregular subgroup of order at least $k$, then either $n\le f_2(k)$, or $G$ is primitive of degree $12^\kappa$ and $G=M_{11}\mathrm{wr}A$, for some positive integer $\kappa$ and for some transitive subgroup $A$ of $\mathrm{Sym}(\kappa)$, where $M_{11}$ is the Mathieu group.

Moreover, if $G$ has no semiregular subgroup of order at least $4$, then one of the following holds
\begin{enumerate}
\item\label{thrmcase1} $n\le 3$,
\item\label{thrmcase2} $n=6$, $G$ is primitive and $G\cong\mathrm{Alt}(5)$,
\item\label{thrmcase3} $n=6$, $G$ is primitive and  $G\cong\mathrm{Alt}(6)$,
\item\label{thrmcase4} $n=36$, $G$ is primitive and $\mathrm{PSU}_3(3)\le G\le \mathrm{P}\Gamma\mathrm{U}_3(3)$,
\item\label{thrmcase5} $n=12^\kappa$, $G$ is primitive and $G=M_{11}\mathrm{wr}A$, for some positive integer $\kappa$ and for some transitive subgroup $A$ of $\mathrm{Sym}(\kappa)$.
\end{enumerate}
\end{theorem}
It was discovered by Giudici~\cite{Giudici} that the Mathieu group $M_{11}$ in its primitive action on $12$ points has no non-identity semiregular elements. Permutation groups having this property are called \textit{\textbf{elusive}} and they are of paramount importance for investigations on the Polycirculant conjecture, see~\cite{Giudici} for details. More generally, Giudici has proved that, for every positive integer $\kappa$ and for every transitive subgroup $A$ of $\mathrm{Sym}(\kappa)$, the group $G= M_{11}\mathrm{wr}A$ endowed with the primitive product action on $12^\kappa$ points is elusive. Therefore, this is a genuine exception in Theorem~\ref{thrm:main2}.

\subsection{Normal coverings and Kronecker classes}\label{coveringsKronecker}
There are some remarkable connections between normal coverings and algebraic number fields, see for instance~\cite{Jehne,Klingen1,Klingen2,Praeger}.

Given an algebraic number field $k$ and a finite extension field $K$ of $k$ the \textit{\textbf{Kronecker set}} of $K$ over $k$ is defined as the set of all prime ideals of the ring of integers of $k$ having a prime divisor of relative degree one in $K$. Then, two finite extensions of $k$ are said to be \textit{\textbf{Kronecker equivalent}} if their Kronecker sets have finite symmetric difference, that is, the Kronecker sets differ only in at most a finite number of primes. This defines an equivalence relation and such extensions are said to belong to the same \textit{\textbf{Kronecker class}}. Clearly, extensions in the same Kronecker class have strong arithmetical similarities.

The connection between problems about Kronecker classes in field extensions and group theoretic problems is explained in~\cite{Jehne,Klingen1,Praeger}. Let $K$ and $K'$ be finite extensions of a given fixed algebraic number field $k$ and let $M$ be a Galois extension of $k$ containing $K$ and $K'$. Let $G=\mathrm{Gal}(M/k)$, $U=\mathrm{Gal}(M/K)$ and $U'=\mathrm{Gal}(M/K')$, in particular, $U$ and $U'$ are the subgroups of $G$ corresponding to $K$ and $K'$ via the Galois correspondence. It is shown in~\cite{Jehne,Klingen1} that $K$ and $K'$ are Kronecker equivalent if and only if
\begin{equation}\label{perlis}
\bigcup_{g\in G}U^g=\bigcup_{g\in G}U'^g.
\end{equation}

This already gives a very strong connection between the problem of understanding Kronecker classes and natural questions in finite permutation groups. For instance, if we consider the permutation representations of $G$ on the right cosets of $U$ and on the right cosets of $U'$, then~\eqref{perlis} is equivalent to the fact  that in these two permutation representations of $G$ the set of derangements is the same.

There is one special case where~\eqref{perlis} yields a natural connection with (weak) normal coverings. Indeed, the special case where $K'/k$ is a Galois extension and $K$ is an extension of $K'$ corresponds to $U\le U'\unlhd G$. In particular, in this special case, $K/k$ and $K'/k$ are Kronecker equivalent if and only if
$$U'=\bigcup_{g\in G}U^g.$$
Using the terminology in~\cite{Praeger}, this yields that $U'$ is a $G$-covering of $U$. As $G$ acts by conjugation as a group of automorphisms on $U'$, when $U'\ne U$, we deduce that $\{U\}$ is a weak normal $1$-covering of $U'$.

There is a number of problems arising in Kronecker classes in algebraic number fields that have been addressed using finite group theory. We report here some open conjectures.
\begin{conjecture}[{{Neumann, Praeger, see~\cite{Praeger2}}}] \label{conjecturePraeger0}
There is an integer function $f$ such that, if $G$ is a finite group with subgroups $U$, $U'$ such that $|G:U'|=n$ and $$\bigcup_{g\in G}U^g=\bigcup_{g\in G}U'^g,$$
then $|G:U|\le f(n)$.
\end{conjecture}
This conjecture phrased in terms of Kronecker classes is as follows.
\begin{conjecture}\label{conjecturePraeger-1}
There is an integer function $f$ such that,  if $K/k$ is an extension of degree $n$ of algebraic number fields and $L/k$ is Kronecker equivalent to $K/k$, then $|L:k| \le f(n)$.
\end{conjecture}

As we mentioned above, with respect to (weak) normal coverings, the case of particular interest is when $U\le U'\unlhd G$.
\begin{conjecture}[{{Neumann, Praeger, see~\cite{Praeger2}}}] \label{conjecturePraeger1}
There is an integer function $g$ such that, if $U'$ is a finite group, $G$ is a group of automorphisms of $U'$ containing the inner automorphisms $\mathrm{Inn}(U')$ as a subgroup of index $n$, and $U$ is a subgroup of $U'$ with
$$\bigcup_{g\in G}U^g=U',$$
then $|G:U|\le g(n)$.
\end{conjecture}

Conjectures~\ref{conjecturePraeger0} and~\ref{conjecturePraeger1} can be phrased in terms of permutations groups: we focus on Conjecture~\ref{conjecturePraeger0}. Let $G,U,U'$ be as in the statement of Conjecture~\ref{conjecturePraeger0} and let $\Omega$ be the set of right cosets of $U$ in $G$. Now, $$\bigcup_{g\in G}U^g$$ is the set of elements of $G$ fixing some element of $\Omega$. If this union equals $\bigcup_{g\in G}U'^g$ and $|G:U'|=n$, then a clique in the derangement graph of $G$ in its action on $\Omega$ has cardinality at most $n$. 
In fact, let $C$ be a clique of size greater than $n$. Then by the pigeonhole principle, $C$ intersects a coset of $U'$ in at least two elements. 
Then the ratio $xy^{-1}$ lies in $U'$ and hence $xy^{-1}$ is conjugate to an element of $U$.  Therefore, $xy^{-1}$ fixes some point, contradicting the fact that $C$ is a clique.
Therefore Conjecture~\ref{conjecturePraeger0} can be seen as a particular case of Question~1.6 in~\cite{KRS}.

In particular, our Theorem~\ref{thrm:main} gives substantial new evidence\footnote{In fact, Theorem~\ref{thrm:main} implies the veracity of Conjecture~\ref{conjecturePraeger0} when the action of $G$ on the right cosets of $U$ is innately transitive.} to the veracity of Conjecture~\ref{conjecturePraeger0} and hence, in turn, to Conjecture~\ref{conjecturePraeger-1} on Kronecker classes.

\section{Proofs of Theorems~\ref{thrm:main} and~\ref{thrm:main1} from Theorem~\ref{thrm:main2}}\label{deduceeasystuff}
In this section, we show that Theorem~\ref{thrm:main2} implies Theorems~\ref{thrm:main} and~\ref{thrm:main1}.

\begin{proof}[Proof of Theorem~$\ref{thrm:main}$]
Let $f_2$ be the function from Theorem~\ref{thrm:main2} and let
\[
f_1(k)=\max(f_2(k),12^{\log_2(k)}).
\]
We show that Theorem~\ref{thrm:main} holds true with this choice of $f_1$.

As semiregular subgroups are cliques in the derangement graph, Theorem~\ref{thrm:main} follows immediately from Theorem~\ref{thrm:main2} using $f_1$, except when $G$ is primitive of degree $12^\kappa$ and $G= M_{11}\mathrm{wr}A$, for some positive integer $\kappa$ and some transitive subgroup $A$ of $\mathrm{Sym}(\kappa)$. Therefore, it suffices to deal with this case.

 Let $\Omega$ be the domain of $G$. Then $\Omega$ admits a Cartesian decomposition $\Delta^\kappa$, where $|\Delta|=12$ and $G$ acts on $\Delta^\kappa$ via its natural primitive product action. In particular, we identify $\Omega$ with $\Delta^\kappa$ and we denote the elements of $G$ as
$$(h_1,\ldots,h_\kappa)a,$$
with $h_1,\ldots,h_\kappa\in M_{11}$ and $a\in A$. Moreover, given $(\delta_1,\ldots,\delta_\kappa)\in \Omega$, we have
$$(\delta_1,\ldots,\delta_\kappa)^{(h_1,\ldots,h_\kappa)a}=(\delta_{1^{a^{-1}}}^{h_{1^{a^{-1}}}},\ldots,\delta_{\kappa^{a^{-1}}}^{h_{\kappa^{a^{-1}}}}).$$

 From Jordan's theorem, $M_{11}$ has a derangement $h$ in its action on $\Delta$. Now, the set
$$\{(h^{\varepsilon_1},h^{\varepsilon_2},\ldots,h^{\varepsilon_\kappa})\in M_{11}^\kappa\mid \varepsilon_1,\ldots,\varepsilon_\kappa\in \{0,1\}\}$$
has cardinality $2^\kappa$ and it is a clique in $\Gamma_G$. In particular, if $2^\kappa\ge k$, then $\Gamma_G$ has a clique of size at least $k$. Otherwise, $2^\kappa<k$ and hence $\kappa<\log_2(k)$. Therefore,
$$|\Omega|=12^\kappa<12^{\log_2(k)}\le f_1(k).$$
Thus, Theorem~\ref{thrm:main} follows also in this case.
\end{proof}

\begin{proof}[Proof of Theorem~$\ref{thrm:main1}$]
Let $G$ be innately transitive and assume that $\Gamma_G$ has no clique of size $4$. Arguing as in the previous proof, we may assume that $G$ is one of the groups appearing in parts~\eqref{thrmcase2}--\eqref{thrmcase5} of Theorem~\ref{thrm:main2}. Moreover, in part~\eqref{thrmcase5} we may assume that $\kappa=1$. We have checked with a computer, using the computer algebra system \texttt{magma}~\cite{magma}, that in the derangement graph of each  these permutation groups there is a clique of size $4$.
\end{proof}

Now, in the rest of this paper, we may focus only on Theorem~\ref{thrm:main2}.
\section{Reduction of Theorem~\ref{thrm:main2} to primitive simple groups}\label{sec:primitive}

We recall that a permutation group $G$ on $\Omega$ is \textit{\textbf{primitive}} if $\Omega$ admits no non-trivial $G$-invariant partition.\footnote{A partition $\pi$ of $\Omega$ is trivial if either each part of $\pi$ has cardinality $1$ and hence $\pi=\{\{\omega\}\mid \omega\in\Omega\}$, or $\pi$ consists of only one part and hence $\pi=\{\Omega\}$.} Moreover, $G$ is said to be \textit{\textbf{quasiprimitive}} if each non-identity normal subgroup of $G$ is transitive on $\Omega$. It is remarkable that these concepts are already present in the work of Galois, see~\cite{pi} for historical details.

Since the orbits of a normal subgroup of a transitive group form a system of imprimitivity, we deduce that each primitive group is quasiprimitive. Moreover, directly from the definition, each quasiprimitive group is innately transitive. Thus, we have the hierarchy
$$\hbox{primitive}\Longrightarrow\hbox{quasiprimitive}\Longrightarrow\hbox{innately transitive.}$$

\begin{lemma}\label{lemma:prel}Let $G$ be an innately transitive group on $\Omega$, let $N$ be a minimal normal subgroup of $G$ transitive on $\Omega$, let $\Sigma$ be a system of imprimitivity, let $\pi:G\to\mathrm{Sym}(\Sigma)$ be the natural homomorphism given by the action of $G$ on $\Sigma$ and let $G^\Sigma$ be the image of $\pi$. If $|\Sigma|>1$ and $\bar X$ is a semiregular subgroup of $G^\Sigma$, then $\pi^{-1}(\bar X)$ is a semiregular subgroup of $G$.
\end{lemma}
\begin{proof}
Let $K=\mathrm{Ker}(\pi)$. Since $N$ is a minimal normal subgroup of $G$, we have $N\le K$ or $K\cap N=1$. If $N\le K$, then $K$ is transitive because so is $N$. Since $K$ acts trivially on $\Sigma$, $K$ fixes setwise each element of $\Sigma$ and, since $|\Sigma|>1$, we deduce that $K$ is intransitive. This contradiction yields $K\cap N=1$. Hence $N$ centralizes $K$. Since $N$ is transitive on $\Omega$, we deduce from~\cite[Theorem~4.2A]{DM} that $K$ is semiregular on $\Omega$.

Let $X=\pi^{-1}(\bar X)$. We prove that $X$ is semiregular on $\Omega$. Indeed, let $\omega\in \Omega$ and let $\sigma\in \Sigma$ with $\omega\in \sigma$. Clearly, $X_\omega\le X_{\sigma}$ because each permutation of $G$ fixing $\omega$ must fix the block of the system of imprimitivity $\Sigma$ containing $\omega$. As $\bar X$ is semiregular on $\Sigma$, we have $\bar{X}_\sigma=1$, that is, $X_\sigma$ fixes setwise each element of $\Sigma$. Therefore, $\pi(X_\sigma)=1$ and $X_\omega\le X_{\sigma}\le \mathrm{Ker}(\pi)=K$. As $K$ is semiregular on $\Omega$, we obtain $X_\omega\le K_\omega=1$. This shows that $X$ is semiregular on $\Omega$.
\end{proof}

The scope of this section is to reduce the proof of Theorem~\ref{thrm:main2} to the case of simple primitive groups. The modern key for analyzing a finite primitive permutation group $G$ is to study
the \textit{\textbf{socle}} $N$ of $G$, that is, the subgroup generated by the minimal normal subgroups of $G$. The socle of a non-trivial finite group is isomorphic to the non-trivial
direct product of simple groups; moreover, for finite primitive groups, these simple groups are pairwise isomorphic. The O'Nan-Scott theorem describes in detail
the embedding of $N$ in $G$ and collects some useful information about the action
of $N$. In~\cite[Theorem]{LPSLPS}, five types of primitive groups are defined (depending
on the group- and action-structure of the socle), namely HA (Affine), AS (Almost
Simple), SD (Simple Diagonal), PA (Product Action) and TW (Twisted Wreath),
and it is shown that every primitive group belongs to exactly one of these types.
We remark that in~\cite{19} this subdivision into types is refined, namely the PA type
in~\cite{LPSLPS} is partitioned in four parts, which are called HS (Holomorphic Simple), HC
(Holomorphic Compound), CD (Compound Diagonal) and PA. For what follows,
we find it convenient to use this subdivision into eight types of the finite primitive
permutation groups.\footnote{This division has the advantage that there are no overlaps between the eight O'Nan-Scott types of primitive permutation groups.}

We start with a technical lemma dealing with the exceptional family involving the Mathieu group arising in Theorem~\ref{thrm:main2}.
\begin{lemma}\label{lemma:technical}
Let $G$ be an innately transitive group on $\Omega$, let $\Sigma$ be a system of imprimitivity such that the permutation group $G^\Sigma$ induced by $G$ on $\Sigma$ is isomorphic to $M_{11}\mathrm{wr} A$ with its natural primitive product action on $12^\kappa$ points, for some positive integer $\kappa$ and some transitive subgroup $A$ of $\mathrm{Sym}(\kappa)$, and let $\pi:G\to\mathrm{Sym}(\Sigma)$ be the natural homomorphism given by the action of $G$ on $\Sigma$. Then  either $G$ in its action on $\Omega$ has semiregular subgroups of order at least $\min(|\mathrm{Ker}(\pi)|\cdot 11^\kappa,660^\kappa)$, or
$\Sigma=\Omega$ and the action of $G=M_{11}\mathrm{wr} A$ on $\Omega$ is the natural primitive product action on $12^\kappa$ points.
\end{lemma}
\begin{proof}
Let  $K$ be the kernel of $\pi$. By definition, $G^\Sigma$ is the image of $\pi$. Let $N$ be a minimal normal subgroup of $G$ with $N$ transitive on $\Omega$: the existence of $N$ is guaranteed by the fact that $G$ is innately transitive on $\Omega$.

 We first assume $K=1$. As $G\cong G^\Sigma=M_{11}\mathrm{wr} A$, we deduce $N$ is the unique minimal normal subgroup of $G$, $G$ is quasiprimitive on $\Omega$ and $N=M_{11}^\kappa$. Let $\sigma\in \Sigma$ and let $\omega\in\sigma$. Since $G^\Sigma$ is endowed with its natural primitive product action of degree $12^\kappa$, we get $$G_\sigma=\mathrm{PSL}_2(11)\mathrm{wr}A\hbox{ and }N_\sigma=\mathrm{PSL}_2(11)^\kappa.$$

If $\Sigma$ is the trivial system of imprimitivity $\{\{\omega\}\mid \omega\in \Omega\}$, then the action of $G=M_{11}\mathrm{wr} A$ on $\Omega$ is the natural primitive product action on $12^\kappa$ points and the lemma is satisfied. Therefore, for the rest of the proof we assume that $\Sigma$ is not the trivial system of imprimitivity. Therefore, $G$ is imprimitive on $\Omega$ and $G_\omega<G_\sigma$. In particular, there exists a maximal subgroup $R$ of $G_\sigma$ with $G_\omega\le R$.
Since there is a one to one order-reversing correspondence between the lattice of subgroups of $G$ containing $G_\omega$ and the systems of imprimitivity for $G$ acting on $\Omega$, $R$ corresponds to the stabilizer of a block $\lambda$ in a  system of imprimitivity, $\Lambda$ say. As $G_\lambda=R\le G_\sigma$, $\Lambda$ is a refinement of the system of imprimitivity $\Sigma$. See Figure~\ref{fig:23}.

Set $H=G_\sigma$ and $\Lambda_\sigma=\{\mu\in \Lambda\mid \mu\subseteq \sigma\}$. We claim that $H$ acts primitively and faithfully on $\Lambda_\sigma$. The fact that $H$ acts primitively on $\Lambda_\sigma$ follows from the fact that, by definition, $R$ is a maximal subgroup of $G_\sigma=H$ and from the fact that $R=G_\lambda$ is the stabilizer of the part $\lambda\in\Lambda_\sigma$ in the system of imprimitivity $\Lambda$. Let $$L=\bigcap_{h\in H}R^h.$$
Observe that $\mathrm{PSL}_2(11)^\kappa$ is the unique minimal normal subgroup of $H$. Therefore, if $H$ were not faithful on $\Lambda_\sigma$, that is $L\ne 1$, then $L$ contains the socle $\mathrm{PSL}_2(11)^\kappa=N_\sigma$ of $H$. Now, since $N$ is transitive on $\Omega$, we have $G=G_\omega N$. Intersecting both sides of this equality with $G_\sigma$ and using the modular law, we deduce $G_\sigma=G_\omega N_\sigma$. Therefore,
$N_\sigma$ seen as a permutation group on $\Omega$ is transitive on the points contained in the block $\sigma$. As $L\ge N_\sigma$, we deduce that $L$ is transitive on the points contained in the block $\sigma$, which is a contradiction because $L\le R=G_\lambda$ fixes setwise the subset $\lambda$ of $\Omega$ and $\lambda\subsetneq \sigma$.

\begin{figure}
\begin{tikzpicture}
\draw[very thick] (-4,-2) rectangle (4,2);
\draw[thick] (-3.9,-1.9)  rectangle (-2.1,1.9);
\draw[thick] (-1.9,-1.9)  rectangle (-0.1,1.9);
\draw[thick] (2.1,-1.9)  rectangle (3.9,1.9);
\draw (-3.8,-1.8)  rectangle (-2.2,-1.1);
\draw (-3.8,-1)  rectangle (-2.2,-.2);
\draw (-3.8,1.8)  rectangle (-2.2,1);
\draw (-1.8,-1.8)  rectangle (-.2,-1.1);
\draw (-1.8,-1)  rectangle (-.2,-.2);
\draw (-1.8,1.8)  rectangle (-.2,1);
\draw (3.8,-1.8)  rectangle (2.2,-1.1);
\draw (3.8,-1)  rectangle (2.2,-.2);
\draw (3.8,1.8)  rectangle (2.2,1);
\fill (.1,0) circle (.5pt);
\fill (.3,0) circle (.5pt);
\fill (.5,0) circle (.5pt);
\fill (.7,0) circle (.5pt);
\fill (.9,0) circle (.5pt);
\fill (1.1,0) circle (.5pt);
\fill (1.3,0) circle (.5pt);
\fill (1.5,0) circle (.5pt);
\fill (1.7,0) circle (.5pt);
\fill (1.9,0) circle (.5pt);
\fill (-1,1.5) circle (1pt);
\fill (-3,0) circle (.5pt);
\fill (-3,0.2) circle (.5pt);
\fill (-3,0.4) circle (.5pt);
\fill (-3,0.6) circle (.5pt);
\fill (-3,0.8) circle (.5pt);
\fill (-1,0) circle (.5pt);
\fill (-1,0.2) circle (.5pt);
\fill (-1,0.4) circle (.5pt);
\fill (-1,0.6) circle (.5pt);
\fill (-1,0.8) circle (.5pt);
\fill (3,0) circle (.5pt);
\fill (3,0.2) circle (.5pt);
\fill (3,0.4) circle (.5pt);
\fill (3,0.6) circle (.5pt);
\fill (3,0.8) circle (.5pt);
\node[below] at (-1,1.5){$\omega$};
\node [below] at (-.4,1){$\lambda$};
\node [below] at (0.1,1.8){$\sigma$};
\node [below] at (4.2,1.8){$\Omega$};
\end{tikzpicture}
\caption{Systems of imprimitivity $\Sigma$ and $\Lambda$: $\Sigma$ is shown with thick lines}\label{fig:23}
\end{figure}
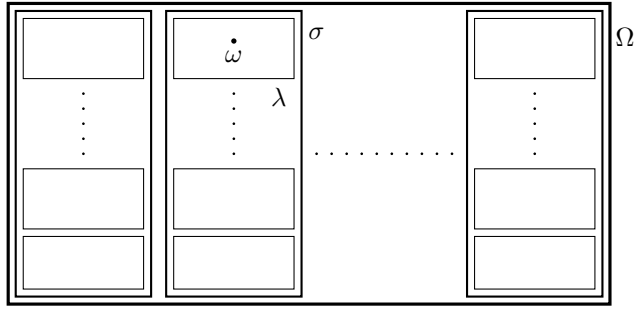

 We apply the O'Nan-Scott theorem to the primitive permutation group $H$ in its action on $\Lambda_\sigma$. 
 Since H, as an abstract group, is isomorphic to $\mathrm{PSL}_2(11)\mathrm{wr} A$, $H$ in its primitive action 
 on $\Lambda_\sigma$ is of type
 \begin{itemize}
 \item AS (when $\kappa=1$), or PA (when $\kappa>1$), or
 \item SD, or CD, or
 \item TW.
 \end{itemize} We deal with each of these cases in turn.

Assume that $H$ in its action on $\Lambda_\sigma$ has type AS or PA. Thus, we have $$G_\lambda=R=B\mathrm{wr}A,$$ for some maximal subgroup $B$ of $\mathrm{PSL}_2(11)$. This shows that $G$ in its action on $\Lambda$ has  stabilizer the wreath product $B\mathrm{wr}A$. Therefore, $\Lambda$ admits a $G$-invariant Cartesian decomposition $\Lambda'^\kappa$, where $\Lambda'$ is the set of right cosets of $B$ in $M_{11}$ and has cardinality $|M_{11}:B|$. Now, $\mathrm{PSL}_2(11)$ has four conjugacy classes of maximal subgroups: isomorphic to $11:5$, $6:2$, and two conjugacy classes isomorphic to $\mathrm{Alt}(5)$. Therefore, $B$ is $M_{11}$-conjugate to one of these five subgroups. We have computed with the auxiliary help of a computer these five permutation representations and we have computed their semiregular subgroups: in the action of $M_{11}$ on the cosets of $11:5$ there are semiregular subgroups of order $144$, in the action of $M_{11}$ on the cosets of $6:2$ there are semiregular subgroups of order $55$, in the action of $M_{11}$ on the cosets of $\mathrm{Alt}(5)$ (for each of the two choices of $M_{11}$-conjugacy classes) there are semiregular subgroups of order $11$. In particular, $M_{11}$ in its action on $\Delta'$ has semiregular subgroups of order at least $11$. Therefore, $G=M_{11}\mathrm{wr}A$ in its action on $\Lambda=\Delta'^\kappa$ has semiregular subgroups of order at least $11^\kappa$. Applying Lemma~\ref{lemma:prel} (with $\Sigma=\Lambda$), we deduce that $G$ in its action on $\Omega$ has semiregular subgroups of order at least $11^\kappa$.

Assume that $H$ in its action on $\Lambda_\sigma$ has type SD or CD. Recall that $N_\sigma=\mathrm{PSL}_2(11)^\kappa$ is the socle of $H$. Let $T_1,\ldots,T_\kappa$ be the $\kappa$ simple direct factors of $N_\sigma$. Then
$(N_\sigma)_\lambda=N_\lambda$ is isomorphic to the direct product of $a$ diagonal subgroups, indeed, up to relabeling the indexed set, there exists a divisor\footnote{When $a=\kappa$, $H$ has type SD, whereas when $1<a<\kappa$, $H$ has type CD.} $a>1$ of $\kappa$ such that
$$N_\lambda=\mathrm{Diag}(T_1\times\cdots\times T_{a})\times \mathrm{Diag}(T_{a+1}\times\cdots \times T_{2a})\times\cdots\times\mathrm{Diag}(T_{\kappa-a+1}\times\cdots\times T_\kappa).
$$
Now, if we let
$$X=\{(x_1,\ldots,x_\kappa)\in N_\sigma=\mathrm{PSL}_2(11)^\kappa\mid x_{ia}=1\, \forall i\in \{1,\ldots,\kappa/a\}\},$$
then we see that $X\cap N_\lambda=1$. Therefore, $X$ acts semiregularly on $\Lambda_\sigma$. What is more,  from the definition of  $X$ we deduce $N_{\lambda^n}\cap X=N_\lambda^n\cap X=1$, $\forall n\in N$. As $N$ is transitive on $\Lambda$, we get that $X$ is semiregular on $\Lambda$. Applying Lemma~\ref{lemma:prel} (with $\Sigma=\Lambda$), we deduce that $G$ in its action on $\Omega$ has semiregular subgroups of order at least $|X|=|\mathrm{PSL}_2(11)|^{\kappa-\kappa/a}=660^{\kappa-\kappa/a}\ge 660^{\kappa/2}\ge 11^\kappa$.

Assume that $H$ in its action on $\Lambda_\sigma$ has type TW. Then the socle $N_\sigma$ of $H=G_\sigma$ acts regularly on $\Lambda_\sigma$, that is, $N_\lambda=1$. As $N_\omega\le N_\lambda=1$, we deduce $N_\omega=1$ and hence $N$ is regular on $\Omega$. So $G$ in its action on $\Omega$ has semiregular subgroups of order at least $|N|=|M_{11}|^{\kappa}\ge 11^\kappa$.

\smallskip

It remains to consider the case that $G$ does not act faithfully on $\Sigma$, that is $K\ne 1$; we pivot on the previous part of the proof. Let $\Lambda$ be the system of imprimitivity consisting of the $K$-orbits, that is, $\Lambda=\{\omega^K\mid \omega\in \Omega\}$. Let $\omega\in \Omega$, let $\sigma\in \Sigma$ with $\omega\in\sigma$ and let $\lambda=\omega^K\in\Lambda$. Observe that the stabilizer of the block $\lambda$ in $G$ is $G_\lambda=KG_\omega$. As $K$ and $G_\omega$ are both subgroups of $G_\sigma$, we get $G_\lambda\le G_\sigma$. Therefore, $\Lambda$ is a refinement of the the system of imprimitivity $\Sigma$. See again Figure~\ref{fig:23}: here the system of imprimitivity $\Lambda$ is formed by the $K$-orbits.

We have
$$K\le \bigcap_{g\in G}(KG_\omega)^g\le \bigcap_{g\in G}G_\sigma^g=\mathrm{Ker}(\pi)=K.$$
Therefore, $K$ is also the kernel of the action of $G$ on $\Lambda$. In particular, applying the first part of the proof with the group $G$ replaced by $G/K$ and with the set $\Omega$ replaced by $\Lambda$, we deduce that either
\begin{itemize}
\item $G/K$ in its action on $\Lambda$ has a semiregular subgroup of order at least $11^\kappa$, or
\item $\Lambda=\Sigma$ and the action of $G/K$ on $\Lambda$ is the natural primitive product action on $12^\kappa$ points.
\end{itemize}
In the first case, Lemma~\ref{lemma:prel} implies that $G$ in its action on $\Omega$ has a semiregular subgroup of order at least $$|K|\cdot 11^\kappa\ge\min(|\mathrm{Ker}(\pi)|\cdot 11^\kappa,660^\kappa).$$
This concludes the analysis of the first case.\footnote{The relevance of $600^\kappa$ in the inequality above arises when dealing with the second case.}

We deal with the second case. Observe that in this case Figure~\ref{fig:23} is somehow misleading, because in this case we have $\Lambda=\Sigma$. Let $M=NK=N\times K$. We have
\begin{align}\label{2}
M&=NM_\omega,\\\label{1}
M_\lambda&=N_\lambda\times K,
\end{align}
where the first equality follows from the fact that $N$ is transitive on $\Omega$ and the second equality follows from the fact that $K$ fixes setwise the $K$-orbit $\lambda=\omega^K$. Since $M_\omega \le M_\lambda=N_\lambda\times K$, we may write each element of $m\in M_\omega$ as an ordered pair $ab$, for a unique $a\in N_\lambda$ and a unique $b\in K$. Let $\pi_{N_\lambda}:M_\omega\to N_\lambda$ and $\pi_K:M_{\omega}\to K$ be the natural projections. From~\eqref{2}, $\pi_K$ is surjective. Let $m\in \mathrm{Ker}(\pi_{N_\lambda})$. Then $m\in K\cap M_\omega=K_\omega=1$, because $K$ is semiregular on $\Omega$. Thus $\pi_{N_\lambda}$ is injective. Moreover, since $M_\lambda$ is transitive on the points contained in the block $\lambda$ and since $\lambda=\omega^K$, from~\eqref{1}, we deduce $$|N_\lambda||K|=|M_\lambda|=|M_\omega||\omega^K|=|M_\omega||K:K_\omega|=|M_\omega||K|.$$ This yields $|N_\lambda|=|M_\omega|$. As $\pi_{N_\lambda}:M_\omega\to N_\lambda$ is injective, we obtain that $\pi_{N_\lambda}$ is surjective and hence it is a bijection.  This shows that $$\psi=\pi_{N_\lambda}^{-1}\circ \pi_K:N_\lambda\to K$$ is a surjective group homomorphism. Furthermore, from the definitions of $\pi_{N_\lambda}$ and $\pi_K$, we have $$M_\omega=\{aa^\psi\mid a\in N_\lambda\}.$$

Recall that, in the case under consideration, $\Lambda=\Sigma$. As $N_\lambda=N_\sigma=\mathrm{PSL}_2(11)^\kappa$ and as $\psi:N_\lambda\to K$ is a surjective group homomorphism, we get $K\cong \mathrm{PSL}_2(11)^\ell$ for some $1\le \ell\le \kappa$, and $\mathrm{Ker}(\psi)=\mathrm{PSL}_2(11)^{\kappa-\ell}$.

Now, if $a\in N_\omega$, then $a\in M_\omega$ and hence $a\in \mathrm{Ker}(\psi)$. Conversely, if $a\in\mathrm{Ker}(\psi)$, then $a=aa^\psi\in M_\omega\cap N=N_\omega$. This yields $\mathrm{PSL}_2(11)^{\kappa-\ell}=\mathrm{Ker}(\psi)=N_\omega$.

Since $N$ is transitive on $\Omega$, we have $G=NG_\omega$. This implies that $G_\omega$ acts transitively by conjugation on the $\kappa$ simple direct factors of $N=M_{11}^\kappa$. Thus $G_\omega$ acts transitively by conjugation on $\kappa$ the simple direct factors of $N_\lambda=\mathrm{PSL}_2(11)^\kappa$. As $N\unlhd G$, we have $N_\omega\unlhd G_\omega$. Putting together the fact that $G_\omega$ acts transitively on the simple direct factors of $N_\lambda=\mathrm{PSL}_2(11)^\kappa$ and the fact that $N_\omega=\mathrm{PSL}_2(11)^{\kappa-\ell}\unlhd G_\omega$, we deduce $\ell=\kappa$. Therefore, $|K|=|\mathrm{PSL}_2(11)^\kappa|$ and hence  $G$ in its action on $\Omega$ has a semiregular subgroup of order at least $|K|=660^\kappa\ge\min(|\mathrm{Ker}(\pi)|\cdot 11^\kappa,660^\kappa).$
\end{proof}

 Using Lemma~\ref{lemma:technical}, we can reduce Theorem~\ref{thrm:main2} to the realm of primitive groups.
\begin{proposition}\label{prop:reduction}
Suppose that Theorem~$\ref{thrm:main2}$ holds true for primitive permutation groups. Then Theorem~$\ref{thrm:main2}$ holds true.
\end{proposition}
\begin{proof}
Let $g:\mathbb{N}\to\mathbb{N}$ be a function witnessing that Theorem~\ref{thrm:main2} holds true for primitive permutation groups. This means that, if $G$ is primitive of degree $n$ and $G$ has no semiregular subgroup of order at least $k$, then either $n\le g(k)$, or $G$ has degree $12^\kappa$ and $G=M_{11}\mathrm{wr}A$, for some positive integer $\kappa$ and for some transitive subgroup $A$ of $\mathrm{Sym}(\kappa)$.  Moreover, if $G$ has no semiregular subgroup of order at least $4$, then one of parts~\eqref{thrmcase1}--\eqref{thrmcase5} holds.

Let $f:\mathbb{N}\to\mathbb{N}$ be the function defined by $$f(k)=\max(g(k)!k,k!(k-1)).$$ We show that the first part of Theorem~\ref{thrm:main2} holds true using this function $f$. Let $G$ be an innately transitive group of degree $n$ and suppose that $G$ has no semiregular subgroup of order at least $k$. If $G$ is primitive, then we have nothing to prove because we are assuming the veracity of Theorem~\ref{thrm:main2} for primitive groups. Therefore, we may suppose that $G$ is imprimitive.

Let $\Omega$ be the domain of $G$ and let $N$ be a minimal normal subgroup witnessing that $G$ is innately transitive, that is, $N$ is transitive on $\Omega$. Let $\Sigma$ be a  system of imprimitivity for the action of $G$ on $\Omega$ with the property that $G$ acts primitively on $\Sigma$.\footnote{The existence of $\Sigma$ is clear: choose a system of imprimitivity whose blocks have cardinality as large as possible.} Let $K$ be the kernel of the action of $G$ on $\Sigma$ and let $G^\Sigma\cong G/K$ be the permutation group induced by $G$ on $\Sigma$. We denote by $\pi:G\to \mathrm{Sym}(\Sigma)$ the natural homomorphism; by definition, $G^\Sigma$ is the image of $\pi$.

Let $\bar X$ be a semiregular subgroup of $G^\Sigma$ and let $X=\pi^{-1}(\bar{X})$ be the preimage of $\bar{X}$ via $\pi$. By Lemma~\ref{lemma:prel}, $X$ is semiregular on $\Omega$.  As $G$ has no semiregular subgroups of order at least $k$, we get
\begin{equation}\label{orderatmostk1}
|\bar{X}||K|=|X|\le k-1.
\end{equation}

As $|K|\ge 1$,~\eqref{orderatmostk1} shows that the primitive group $G^\Sigma$ has no semiregular subgroups of order at least $1+(k-1)/|K|\le k$. Since we are assuming the veracity of Theorem~\ref{thrm:main2} for primitive permutation groups, we deduce that either
\begin{itemize}
\item $|\Sigma|\le g(k)$, or
\item $G^\Sigma$ is primitive of degree $12^\kappa$  and $G^\Sigma=M_{11}\mathrm{wr}A$, for some  positive integer $\kappa$ and for some transitive subgroup $A$ of $\mathrm{Sym}(\kappa)$.
\end{itemize}
Assume first that $|\Sigma|\le g(k)$. In particular,
$$|G^\Sigma|\le |\mathrm{Sym}(g(k))|\le g(k)!.$$ Since $n\le |G|$,  we deduce
\begin{equation*}
n\le |G|= |G:K||K|=|G^\Sigma||K|\le g(k)!k\le f(k),
\end{equation*}
where we are using~\eqref{orderatmostk1} in the second inequality.

Assume that the second possibility above holds. Our auxiliary Lemma~\ref{lemma:technical} implies that either $G$ in its action on $\Omega$ has a semiregular subgroup of order at least $\min(|K|\cdot 11^\kappa,660^\kappa)\ge 11^\kappa$, or
$\Sigma=\Omega$ and the action of $G=M_{11}\mathrm{wr} A$ on $\Omega$ is the natural primitive product action on $12^\kappa$ points. The second case is impossible in our situation because we are assuming that $G$ is imprimitive on $\Omega$. Moreover, if $k\le 11^\kappa$, then $G$ does have a semiregular subgroup of order at least $k$. Assume then $11^\kappa<k$. Observe that $G^\Sigma=M_{11}\mathrm{wr}A$ has a faithful permutation representation of degree $11^\kappa$ and hence $|G^\Sigma|\le (11^\kappa)!\le k!$. Therefore, $$n=|\Omega|\le |G|=|G^\Sigma||K|\le k!(k-1)\le f(k),$$
where as above we are using~\eqref{orderatmostk1} in the second inequality.

\smallskip
It remains to prove the second part of the statement of Theorem~\ref{thrm:main2} for innately transitive groups. Therefore, let $G$ be innately transitive with no semiregular subgroups having order at least $4$. We use the notation above (with $k=4$). In particular, we may assume that $G$ is not primitive, because we are assuming the veracity of Theorem~\ref{thrm:main2} for primitive groups. Recall that $G^\Sigma$ is primitive and either
\begin{itemize}
\item each semiregular subgroup of $G^\Sigma$ has order at most  $(k-1)/|K|=3/|K|$, see~\eqref{orderatmostk1} with $k=4$, or
\item $G^\Sigma$ has degree $12^\kappa$  and $G^\Sigma=M_{11}\mathrm{wr}A$, for some  positive integer $\kappa$ and for some transitive subgroup $A$ of $\mathrm{Sym}(\kappa)$.
\end{itemize}

Assume that the second possibility above holds. Our auxiliary Lemma~\ref{lemma:technical} implies that either $G$ in its action on $\Omega$ has a semiregular subgroup of order at least $11^\kappa$, or
$\Sigma=\Omega$ and the action of $G=M_{11}\mathrm{wr} A$ on $\Omega$ is the natural primitive product action on $12^\kappa$ points. In the first case we have a semiregular subgroup of order at least $11\ge 4$ and in the second case we obtain that part~\eqref{thrmcase5} holds. This concludes the proof for this case.

Assume now that the first possibility above holds. In particular, as $3/|K|<4$ and as $G^\Sigma$ is primitive, by hypothesis, one of parts~\eqref{thrmcase1}--\eqref{thrmcase5} holds for $G^\Sigma$. Observe that part~\eqref{thrmcase5} is exactly the second possibility, which we have already dealt with; therefore, we may disregard this part from further consideration. Before dealing with each of the remaining four possibilities we make some preliminary observations.

Recall that $|K|\le 3$, because $K$ is semiregular on $\Omega$ and $G$ has no semiregular subgroups of order at least $4$. Assume $|K|\in\{2,3\}$. Then $3/|K|<2$ and  hence $G^\Sigma$ has no non-trivial semiregular subgroups. A direct inspection on parts~\eqref{thrmcase1}--\eqref{thrmcase4} shows that $|\Sigma|=1$ and $G^\Sigma=1$. This is impossible because $\Sigma$ is a non-trivial system of imprimitivity of $\Omega$ and hence $|\Sigma|>1$.

Assume $|K|=1$. In particular, $G\cong G^\Sigma$ as abstract groups. We have constructed with a computer the abstract group $G$ (for each of the cases arising in parts~\eqref{thrmcase1}--\eqref{thrmcase4}), we have determined all the imprimitive innately transitive faithful actions of these groups and we have verified that in each action $G$ admits a semiregular subgroup of order at least $4$.
\end{proof}

In the light of Proposition~\ref{prop:reduction}, for the rest of the paper we may focus only on the class of primitive groups.  In the rest of the section, we reduce the proof of Theorem~\ref{thrm:main2} further, indeed to the case of primitive simple groups.

For six of the eight O'Nan-Scott types, the proof of Theorem~\ref{thrm:main2} is immediate: the socle of $G$ contains a subgroup acting regularly on the domain and hence forming a clique in the derangement graph.
\begin{lemma}\label{basic}
Let $G$ be a primitive group of degree $n$ of type HA, HS, HC, TW, SD or CD. Then $G$ has a semiregular subgroup of order $n$. In particular, Theorem~$\ref{thrm:main2}$ holds true in these cases.
\end{lemma}
\begin{proof}
Let $N$ be the socle of $G$ and let $\Omega$ be the domain of $G$.

When $G$ is of type HA or TW, $N$ acts regularly on $\Omega$. In particular, $N$ is a clique in $\Gamma_G$ of cardinality $|N|=|\Omega|=n$.

When $G$ is of type HS or HC, $N$ is the direct product of two minimal normal subgroups of $G$, say $M_1$ and $M_2$. From the description of the primitive groups of type HS or HC, we see that $M_1$ and $M_2$ act regularly on $\Omega$ and they form a clique in $\Gamma_G$ of cardinality $|M_i|=|\Omega|=n$.

Suppose $G$ is of type SD. Then $N=T_1\times \cdots \times T_{\ell+1}$, where $T_1,\ldots,T_{\ell+1}$ are pairwise isomorphic non-abelian simple groups. From the description of the primitive groups of type SD, we see that $|\Omega|=|T_1|^\ell$ and that $T_1\times\cdots \times T_\ell$ acts regularly on $\Omega$.  Therefore, as above, $T_1\times\cdots \times T_\ell$ forms a clique in $\Gamma_G$ of cardinality $|\Omega|=n$.

Suppose that $G$ is of type CD. Then $\Omega$ admits a non-trivial Cartesian decomposition, that is, $\Omega=\Delta^\kappa$ for some finite set $\Delta$ and for some positive integer $\kappa\ge 2$, and we have an embedding $G\le H\mathrm{wr}\mathrm{Sym}(\kappa)$, where the wreath product $H\mathrm{wr}\mathrm{Sym}(\kappa)$ acts on $\Delta^\kappa$ primitively, $H\le\mathrm{Sym}(\Delta)$ and $H$ is of type SD in its action on $\Delta$. Now, if the socle of $H$ is isomorphic to $T^{\ell+1}$, for some non-abelian simple group $T$ and for some positive integer $\ell\ge 1$, then the socle of $G$ is isomorphic to $T^{\kappa(\ell+1)}$. In particular, the socle of $G$ contains a subgroup isomorphic to $T^{\kappa\ell}$ acting regularly on $\Omega$ and we may argue as above.
\end{proof}
In the light of Lemma~\ref{basic} it is clear that the bulk of the argument for proving Theorem~\ref{thrm:main2} is dealing with primitive groups of AS and PA type. For dealing with these two cases, we require detailed information on non-abelian simple groups. We conclude this section with a reduction to primitive simple groups.

\begin{proposition}\label{prop:reduction2}
Suppose that Theorem~$\ref{thrm:main2}$ holds true for primitive simple groups. Then Theorem~$\ref{thrm:main2}$ holds true.
\end{proposition}
\begin{proof}
In view of Proposition~\ref{prop:reduction}, we may suppose that $G$ is primitive. Moreover, in view of Lemma~\ref{basic}, we may suppose that $G$ is of AS or PA type.

Let $\Omega$ be the domain of $G$. Then $\Omega$ admits a Cartesian decomposition $\Delta^\kappa$, for some $\kappa\ge 1$\footnote{When $\kappa=1$, $\Omega=\Delta$ and $G$ is of type AS, when $k\ge 2$, $G$ is of type PA.}  and $G$ embeds into the wreath product $H\mathrm{wr}\mathrm{Sym}(\kappa)$ endowed with the primitive product action. Replacing $\mathrm{Sym}(\kappa)$ by a suitable transitive subgroup $A$, we may suppose that $G$ embeds into the wreath product $H\mathrm{wr}A$ and $G$ projects surjectively to $A$. Moreover, $H$ is of type AS. Let $T$ be the socle of $H$. Then the socle of $G$ is $T^\kappa$. When $T=M_{11}$ and $|\Delta|=12$, as we have mentioned in the introduction, Giudici~\cite{Giudici} has shown that $G$ has no non-identity semiregular element and hence, for the rest of the argument, we may suppose that $T$ is not $M_{11}$ in its degree $12$ action.

Observe that $T$ acts transitively on $\Delta$ because $H$ is primitive on $\Delta$, but $T$ is not necessarily primitive on $\Delta$. Let $\Sigma$ be a non-trivial system of imprimitivity for the action of $T$ on $\Delta$; by choosing the blocks of $\Sigma$ as large as possible, we may assume that $T$ acts primitively on $\Sigma$.

We now prove the first part of the statement of Theorem~\ref{thrm:main2}. Let $g:\mathbb{N}\to\mathbb{N}$ be a function witnessing that Theorem~\ref{thrm:main2} holds for primitive simple groups. Without loss of generality we may suppose that $g(1)=1$. Define $f:\mathbb{N}\to\mathbb{N}$ by $$f(k)=\max\{g(\lfloor k^{1/\ell} \rfloor)!^\ell\mid \ell\in \mathbb{N}\}.$$
Observe that $f$ is well-defined because when $\ell\ge k$, we have $\lfloor k^{1/\ell}\rfloor=1$ and hence $g(\lfloor k^{1/\ell}\rfloor)!^\ell=1$. In particular,
$$f(k)=\max\{g(\lfloor k^{1/\ell} \rfloor)!^\ell\mid \ell\in \{1,\ldots,k\}\}.$$

 Let $k\in\mathbb{N}$. By hypothesis, either $T$ in its action on $\Sigma$ has a semiregular subgroup $X$ of order at least $k^{1/\kappa}$, or $|\Sigma|\le g(\lfloor k^{1/\kappa}\rfloor)$. Observe that $X$ is also semiregular for the action of $T$ on $\Delta$. Therefore, in the first case, $X^\kappa$ is a semiregular subgroup of $G$ of order at least $(k^{1/\kappa})^\kappa=k$. Assume then $|\Sigma|\le g(\lfloor k^{1/\kappa}\rfloor)$. Thus $|\Delta|\le |T|\le |\Sigma|!\le g(\lfloor k^{1/\kappa}\rfloor)!$ and $|\Omega|=|\Delta|^\kappa\le g(\lfloor k^{1/\kappa}\rfloor)!^\kappa\le f(k)$.

We now prove the second part of the statement of Theorem~\ref{thrm:main2}. Therefore, we suppose that $G$ has no semiregular subgroup of order at least $4$. We use the notation established above. Assume first $\kappa=1$, that is, $G$ is almost simple. Since we are assuming that Theorem~\ref{thrm:main2} holds for simple primitive groups, we deduce that $T$ is isomorphic to $\mathrm{Alt}(5)$, $\mathrm{Alt}(6)$, $M_{11}$ or $\mathrm{PSU}_3(3)$. We have
constructed with a computer the abstract group $G$ having socle $T$, we have determined all the primitive
actions of these groups and we have verified the veracity of Theorem~\ref{thrm:main2}. Assume next $\kappa\ge 2$. If $T$ has a semiregular subgroup of order at least $2$, then $T^\kappa\le G$ has a semiregular subgroup of order at least $4$. An inspection on the cases arising in parts~\eqref{thrmcase2}--\eqref{thrmcase5}, we see that the only group not having a semiregular subgroup of order at least $2$ is $T=M_{11}$ in its primitive action of degree $12$, which we have already dealt with above.
\end{proof}

In view of Proposition~\ref{prop:reduction2}, for the proof of Theorem~\ref{thrm:main2} we may suppose that $G$ is simple and primitive.

\section{Number theoretic results}\label{sec:number theory}
We collect in this section some number theoretic results. Remarkably the proof of Theorem~\ref{thrm:main2} relies on some deep number theoretic facts, most notably, a quantitative weak version of the $abc$ conjecture due to Stewart and Tijdeman~\cite{abc} for dealing with alternating groups, and an impressive theorem of Siegel~\cite{Siegel} on the greatest prime factors of polynomials valuated at integers for dealing with simple groups of Lie type\footnote{We would like to thank the
pseudonymous user “so-called friend Don” who directed us to~\cite{Siegel} in response to a question we posed on MathOverflow.}.

Given a prime number $p$ and a non-negative integer $x$, we let $x_p$ denote the remainder of $x$ in the division by $p$.
We need the famous theorem of Sylvester on prime numbers\footnote{We thank Marina Cazzola for pointing out the relevance of~\cite{eps} in our work.}, see for instance~\cite{Hall}.
\begin{theorem}\label{lemma:sylvester}Let $\ell$ be a positive integer. Then the product of $\ell$ consecutive integers greater than $\ell$ is divisible by a prime $p$ greater than $\ell$.
\end{theorem}

Let $q$ and $t$ be positive integers. Recall that a \textit{\textbf{primitive prime divisor}} for the pair $(q,t)$ is a prime $p$ such that $p\mid q^t-1$ and $p\nmid q^i-1$, for all $1\le i<t$. Zsigmondy's theorem~\cite{zi} shows that $q^t-1$ admits a primitive prime divisor, except when $t=2$ and $q=2^m-1$ is a Mersenne number, or when $(t,q)=(6,2)$.

\begin{lemma}\label{number theory}
Let $m$ be a positive integer with $m\ge 5$ and let $\ell\in \{1,\ldots,m-1\}$. Suppose that, for every prime $p\ge 5$, $\ell_p\le m_p$. Then either
\begin{enumerate}
\item\label{case1} $\ell\in \{1,m-1\}$ and $m=2^a\cdot 3^b$, for some $a,b\in\mathbb{N}$, or
\item\label{case2}$m=9$ and $\ell\in \{2,7\}$.
\end{enumerate}
\end{lemma}
\begin{proof}
Suppose that $m$ and $\ell$ satisfy the property:
$$(\dag)\quad \ell_p\le m_p, \hbox{ for every prime }p\ge 5.$$ Now, consider $\ell'=m-\ell$ and let $p\ge 5$ be a prime number. By hypothesis $\ell_p\le m_p$ and hence $m_p-\ell_p$ is the remainder of $\ell'=m-\ell$ in the division by $p$, that is, $\ell'_p=(m-\ell)_p=m_p-\ell_p\le m_p$. This shows that, if  the pair $(m,\ell)$ satisfies $(\dag)$, then so does $(m,\ell')=(m,m-\ell)$. Therefore, without loss of generality, replacing $\ell$ by $m-\ell$ if necessary, we may suppose that $\ell\le m/2$.

 Now consider the $\ell$ consecutive numbers
$$m,m-1,\ldots,m-\ell+1.$$
As $m\ge 2\ell$, these numbers are greater than $\ell$ and hence, by Sylvester's theorem, there exists a prime
\begin{equation}\label{prime p}p>\ell\end{equation} dividing $m-i$, for some $i\in \{0,\ldots,\ell-1\}$. As $i\le \ell-1$ and $p\mid m-i$, we have $m_p\le i\le \ell-1$. However, as $p>\ell$, we have $\ell_p=\ell$ and hence we deduce $m_p<\ell_p$. Since $(m,\ell)$ satisfies $(\dag)$, we have $p<5$, that is, $p\in \{2,3\}$.

Assume $\ell=1$. If $m$ is divisible by a prime $p\ge 5$, then $\ell_p=1>m_p=0$ and hence $(m,\ell)$ does not satisfy $(\dag)$. Therefore $m=2^a\cdot 3^b$, for some $a,b\in\mathbb{N}$, and we obtain part~\eqref{case1}.

Assume $\ell\ge 2$. From~\eqref{prime p}, we have $p>\ell\ge 2$. As $p\in \{2,3\}$, we deduce $p=3$ and, more importantly, $\ell=2$. For each prime divisor $r\ge 5$ of $m-1$ or $m$, we have $m_r\le 1$ and hence $\ell_r\le m_r\le 1$, because $(m,\ell)$ satisfies $(\dag)$. As $\ell=2$, the condition $\ell_r\le 1$ can only be satisfied if and only if $m-1$ and $m$ are only divisible by the primes $2$ and $3$. Thus $$m=2^a\cdot 3^b\hbox{ and }m-1=2^{a'}\cdot 3^{b'},$$ for some $a,a',b,b'\in\mathbb{N}$. Since $$1=\gcd(m,m-1)=\gcd(2^a\cdot 3^a,2^{a'}\cdot 3^{b'})=2^{\min(a,a')}\cdot 3^{\min(b,b')},$$ we obtain
\begin{itemize}
\item $m=2^a$ and $m-1=3^{b'}$, or
\item $m=3^b$ and $m-1=2^{a'}$.
\end{itemize}
We deal with each of these two cases in turn. Assume $m=2^{a}$ and $m-1=3^{b'}$. As $m\ge 5$, $3$ divides $m-1=2^a-1$ and hence  $a$ is even. Thus $a=2\alpha$ for some integer $\alpha$. This gives $2^{a}-1=4^{\alpha}-1=3^{b'}$. In particular, since $3$ divides $4^1-1$,  $4^{\alpha}-1$ has no primitive prime divisors. Using the theorem of Zsigmondy~\cite{zi}, we deduce that this case is impossible unless $m=4$. However, this contradicts our hypothesis $m\ge5$.  Assume $m=3^{b}$ and $m-1=2^{a'}$. Thus $3^{b}-1=2^{a'}$. In particular, since $2$ divides $3^1-1$,  $3^{b}-1$ has no primitive prime divisors. Using the theorem of Zsigmondy~\cite{zi}, we deduce that $b\in \{1,2\}$. When $b=1$, $m=3$ and we contradict our hypothesis $m\ge5$. When $b=2$, $m=9$ and we obtain the exceptional case in~\eqref{case2}.
\end{proof}

\begin{lemma}\label{lemma:number theory2}Let $m$ be a positive integer. If $m\ge 8$, then there exists a prime $p$ with $m/2<p\le m-3$.
\end{lemma}
\begin{proof}
Bertrand's postulate~\cite[page 498]{Hardy} says that, when $n\ge 4$, there is a prime $p$ satisfying $n<p<2n-2$.

In particular, when  $m$ even, the proof follows by applying Bertrand's postulate with $n=m/2$. Whereas, when $m$ is odd, the proof follows by applying Bertrand's postulate with $n=(m-1)/2$.
\end{proof}

\begin{lemma}\label{lemma:number theory3}Let $m$ be a positive integer. Then $(m/2)^m\ge  m!/2$.
\end{lemma}
\begin{proof}
This follows from an inductive argument on $m$.
\end{proof}

The \textit{\textbf{radical}} $\mathrm{rad}(m)$ of a positive integer $m$ is  the product of the distinct prime numbers dividing $m$, that is,
\[\mathrm{rad}(m)=\prod_{\substack{p\mid m\\p\textrm{ prime}}}p.  \]
 For instance, $\mathrm{rad}(24)=2\cdot 3=6$. In our work, we need the following weak form of the abc conjecture, see~\cite{abc}.
\begin{theorem}\label{abctheorem}
There exists a positive constant $\kappa$ such that, if $a$, $b$ and $c$ are coprime positive integers with $c=a+b$, then $c\le \exp(\kappa\cdot \mathrm{rad}(abc)^{15})$.
\end{theorem}

	 Following~\cite{Siegel}, given a positive integer $n$, we denote by $P[n]$ the \textit{\textbf{greatest prime factor}} of the integer $n$. As customary, we denote with $\Phi_n(x)\in\mathbb{Z}[x]$ the $n^{\mathrm{th}}$ cyclotomic polynomial, that is,
$$\Phi_n(x)=\prod_{\substack{\zeta \textrm{ primitive }n^{\mathrm{th}}\\ \textrm{root of unity}}}x-\zeta.$$	
\begin{lemma}\label{siegeltheorem}
Given $n\in\mathbb{N}$ with $n\ge 3$ and $q\in\mathbb{N}$ with $q\ge 2$, there exist two positive constants $c$ and $c'$ depending on $n$ only such that $q^n-1$ admits a primitive prime divisor $p\ge c\log \log q$ for every $q\ge c'$.
\end{lemma}	
\begin{proof}
This follows from a remarkable result of Siegel~\cite[Satz~7]{Siegel}: Let $f\in\mathbb{Z}[x]$ be a polynomial with integer coefficients and at least $2$ distinct roots. Then there exist two positive constants  $c_f$ and $c_f'$ depending on $f$ only such that $P[f(q)]\ge c_f\log\log q$, $\forall q\in \mathbb{N}$ with $q\ge c_f'$.

As $n\ge 3$, $\Phi_n(x)$ has $\varphi(n)\ge 2$ distinct roots and hence we may apply Siegel's theorem with $f(x)=\Phi_n(x)$. In particular, there exist two positive constants $c$ and $c'$ depending on $n$ only such that $P[\Phi_n(q)]\ge c\log \log q$, for every integer $q$ with $q\ge c'$. Replacing $c'$ by a larger constant, we may also suppose that $P[\Phi_n(q)]\ge n+1$, for every integer $q\ge c'$.
Let $q\in\mathbb{N}$ with $q\ge c'$ and let $p=P[\Phi_n(q)]$.

Following~\cite[Definition~1]{Glasby}, we let $\Phi_n^\ast(q)$ denote the largest divisor of $\Phi_n(q)$ relatively prime to $$\prod_{i=1}^{n-1}q^i-1.$$Let $r$ be the largest prime divisor of $n$.
From~\cite[Lemma~3.1]{Glasby}, we have \[
\Phi_n^*(q)=
\begin{cases}
\Phi_n(q)&\textrm{if }r\textrm{ does not divide }\Phi_n(q),\\
\Phi_n(q)/r&\textrm{if }r\textrm{ divides }\Phi_n(q).
\end{cases}
\]
Since $p>n\ge r$, we deduce that $p$ divides $\Phi_n^\ast(q)$ and hence, by definition, $p$ is a primitive prime divisor of $q^n-1$.
\end{proof}

\section{Alternating groups and Sporadic groups}\label{thrm:alt}
In this section we prove Theorem~\ref{thrm:main2} when $G=\mathrm{Alt}(m)$ is an alternating group of degree $m\ge 5$ and when $G$ is a sporadic simple group.

We start by dealing with the alternating group $G=\mathrm{Alt}(m)$ with $m\ge 5$. Let $\Omega$ be a $G$-set with $G$ acting faithfully and transitively on $\Omega$ and let $\omega\in \Omega$. As the point stabilizer $G_\omega$ is a subgroup of $G=\mathrm{Alt}(m)$, we deduce that $G_\omega$ acts on the set $\{1,\ldots,m\}$. Now, we consider three cases, depending on whether $G_\omega$ in its action on $\{1,\ldots,m\}$ is intransitive, imprimitive (that is, transitive but not primitive), or primitive. As usual, we let $n=|\Omega|$.

\begin{lemma}\label{Alt:type1}
There exists a function $f:\mathbb{N}\to \mathbb{N}$ such that, if $G_\omega$ is intransitive on $\{1,\ldots,m\}$, then either $G$ in its action on $\Omega$ has a semiregular subgroup of order at least  $k$ or $n\leq f(k)$.  Moreover, either $G$ in its action on $\Omega$ has a semiregular subgroup of order at least $4$, or $m=|\Omega|=6$.
\end{lemma}

\begin{proof}
Let $\kappa$ be the absolute constant arising in Theorem~\ref{abctheorem} and let $f:\mathbb{N}\to\mathbb{N}$ be defined by
$$f(k)=\max\{2k,(\exp(\kappa\cdot (k-1)^{15(k-1)}))^k\}.$$

As $G_\omega$ is intransitive on $\{1,\ldots,m\}$, $G_\omega$ fixes setwise a subset $L$ of $\{1,\ldots,m\}$ having cardinality $\ell$, for some positive integer $\ell$ with $1\le \ell\le m/2$. Thus $G_\omega\le G\cap(\mathrm{Sym}(L)\times \mathrm{Sym}(\{1,\ldots,m\}\setminus L))$. As $G$ is primitive on $\Omega$, we deduce $G_\omega= G\cap(\mathrm{Sym}(L)\times \mathrm{Sym}(\{1,\ldots,m\}\setminus L))$. Hence we may identify $\Omega$ with the set of $\ell$-subsets of $\{1,\ldots,m\}$ and we may identify the action of $G$ on $\Omega$ with the natural action of $\mathrm{Alt}(m)$ on $\ell$-subsets.

Assume first $\ell=1$. In this case, $n=m$ and the action of $G$ on $\Omega$ is the natural action of the alternating group $\mathrm{Alt}(m)$ of degree $m$. In particular, $G$ in its action on $\Omega$ has a semiregular subgroup of order $m$ when $m$ is odd, and of order $m/2$ when $m$ is even. Thus, when $m/2\ge k$ we can guarantee the existence of a semiregular subgroup of sufficiently large cardinality and, when $m/2<k$, we have $n=m\le f(k)$. For the rest of the argument, we may assume $\ell\ge 2$.

Suppose
\begin{center}
$(\dag)\qquad$ there exists a prime divisor $p$ of $m(m-1)\cdots(m-\ell+1)$ with $p\ge \max(k,\ell+1)$.
\end{center}
As $p$ is prime, there exists $i\in \{0,...,\ell-1\}$ with $p\mid m-i$. As $i\le \ell-1<p$, we have
	\[m_p=i\leq \ell-1.\]
Let $x\in \mathrm{Alt}(m)=G$ be a permutation having (in its action on $\{1,\ldots,m\}$) $m_p$ fixed points and $(m-m_p)/p$ disjoint cycles of length $p$. We claim that $X=\langle x\rangle$ is a semiregular subgroup of $G$ in its action on $\Omega$. Indeed, as $X$ has prime order $p$, as $p>\ell$ and $m_p<\ell$, no non-identity element of $X$ fixes setwise any $\ell$-subset. This shows that, when $(\dag)$ holds, $G$ in its action on $\Omega$ has a semiregular subgroup of order at least $p\ge k$.

Assume $k\le \ell+1$. As $\ell\le m/2$, by Theorem~\ref{lemma:sylvester},  there exists a prime divisor $p$ of $m(m-1)\cdots(m-\ell+1)$ with $p>\ell$. Thus $p\ge \ell+1\ge k$ and hence $p\ge\max(k,\ell+1)$. Therefore, in this case,~$(\dag)$ is satisfied.

Assume  $k\ge \ell+2$. Suppose there exists a prime divisor $p$ of $m(m-1)\cdots(m-\ell+1)$ with $p\ge k$. Since $p\ge k\ge \ell+2$,~$(\dag)$ is satisfied. Finally suppose that there exists no prime divisor $p$ of $m(m-1)\cdots(m-\ell+1)$ with $p\ge k$.
 In particular, as $\ell\ge 2$, all primes dividing $m(m-1)$ are smaller than $k$. Since the number of primes less than $k$ is at most $k-1$, we deduce
$$\mathrm{rad}(m(m-1))\le \prod_{\substack{p \textrm{ prime}\\p\le k-1}}p\le (k-1)^{k-1}.$$
 Using Theorem~\ref{abctheorem} with $a=1$, $b=m-1$ and $c=m$, we get
\[m\leq \exp(\kappa\cdot (k-1)^{15(k-1)}).\]
Thus
\[n=|\Omega|=\binom{m}{\ell}\le m^\ell\le m^k\leq (\exp(\kappa\cdot (k-1)^{15(k-1)}))^k\le f(k).\]

To conclude the proof, we need to discuss the existence of semiregular subgroups of order at least $k=4$. If~$(\dag)$ is satisfied with $k=4$, then we have semiregular subgroups of order at least $5$. Notice that, when $\ell\ge 3$, Theorem~\ref{lemma:sylvester} guarantees that $(\dag)$ is satisfied with $k=4$. Indeed, there exists a prime divisor of $m(m-1)\cdots (m-\ell+1)$ with $p\ge \ell+1$. As $p$ is prime, when $\ell\ge 3$, we have $p\ge 4$ and hence $p\ge 5$. Therefore, we may suppose that $\ell\le 2$. When $\ell=1$, $G$ has a semiregular subgroup of order $m/2$ if $m$ is even and $m$ if $m$ is odd. These values are less than $4$ only when $m=6$; therefore, we obtain the exceptional case listed in the statement of the lemma. Finally assume $\ell=2$ and suppose that $(\dag)$ is not satisfied with $k=4$. Notice that, for every prime divisor $p\ge 5$, we have $2=\ell_p\le m_p$.  Then by Lemma~\ref{number theory} we get $m=9$ and $\ell=2$.  When $m=9$ and $\ell=2$, observe that a cyclic subgroup of $\mathrm{Alt}(9)$ of order $9$ acts semiregularly on the $2$-subsets of $\{1,\ldots,9\}$.
\end{proof}

\begin{lemma}\label{Alt:type2}
Suppose that $G_\omega$ is imprimitive on $\{1,\ldots,m\}$. Then $G$ in its action on $\Omega$ has a semiregular subgroup of order $k$ with $n\le k^{2k}$. Moreover, $G$ in its action on $\Omega$ has a semiregular subgroup of order at least $4$.
\end{lemma}

\begin{proof}
As $G_\omega$ is imprimitive on $\{1,\ldots,m\}$, $G_\omega$ admits a non-trivial system of imprimitivity with $a$ blocks of cardinality $b$, for some positive integers $a$ and $b$ with $1<a,b<m$ and $m=ab$. Therefore, $G_\omega$ embeds into the imprimitive wreath product $\mathrm{Sym}(b)\mathrm{wr}\mathrm{Sym}(a)$.

From Lemma~\ref{lemma:number theory2}, there exists a prime $p$ with $m/2<p\le m$. In particular, $p$ is relatively prime to $|G_\omega|$, because $p$ does not divide $b!^aa!$.
Therefore, a cyclic subgroup of order $p$ of $G=\mathrm{Alt}(m)$ acts semiregularly on $\Omega$. From Lemma~\ref{lemma:number theory3}, we have $p^{2p}\ge m!/2=|G|\ge |\Omega|=n$. Observe that $p\ge 5$ and hence $G$ contains a semiregular subgroup of order at least $4$.
\end{proof}

\begin{lemma}\label{Alt:type3}
Suppose that $G_\omega$ is primitive on $\{1,\ldots,m\}$. Then $G$ in its action on $\Omega$ has a semiregular subgroup of order $k$ with $n\le k^{2k}$. Moreover, $G$ in its action on $\Omega$ has a semiregular subgroup of order at least $4$ unless one of the following holds
\begin{enumerate}
\item\label{Alt1case1}$m=5$ and $|G_\omega|=10$,
\item\label{Alt1case2}$m=6$ and  $|G_\omega|=60$.
\end{enumerate}
\end{lemma}
\begin{proof}
Assume $m\ge 8$. From Lemma~\ref{lemma:number theory2}, there exists a prime $p$ with $m/2<p\le m-3$. If $p$ divides $|G_\omega|$, then $G_\omega$ contains a cycle of length $p$ in its action on $\{1,\ldots,m\}$. From a classical result of Jordan~\cite[Theorem~3.3E]{DM}, we deduce $G_\omega\ge\mathrm{Alt}(m)$, which contradicts the fact that $G$ acts faithfully on $\Omega$. Therefore, $p$ is relatively prime to $|G_\omega|$. In particular, a cyclic subgroup of order $p$ of $G=\mathrm{Alt}(m)$ acts semiregularly on $\Omega$. From Lemma~\ref{lemma:number theory3}, we have $p^{2p}\ge m!/2=|G|\ge |\Omega|=n$. Observe also that $p\ge 5$ and hence $G$ contains a semiregular subgroup of order at least $4$.

Assume now $m<8$. Here the proof follows from a computer computation with the invaluable help of the computer algebra system \texttt{magma}~\cite{magma}.
\end{proof}

\begin{corollary}\label{cor:alt}
There exists a function $f:\mathbb{N}\to\mathbb{N}$ such that, if $k$ is a positive integer and  $G$ is an alternating group $\mathrm{Alt}(m)$ with $m\ge 5$ acting  faithfully and transitively on a set of cardinality $n$, then either $G$ has a semiregular subgroup of order at least $k$, or $n\le f(k)$. Moreover, $G$ has a semiregular subgroup of order at least $4$ unless one of the following holds
\begin{enumerate}
\item\label{cAlt1case1}$m=5$, $G_\omega$ is primitive on $\{1,\ldots,m\}$ and $|G_\omega|=10$,
\item\label{cAlt1case2}$m=6$, $G_\omega$ is primitive on $\{1,\ldots,m\}$ and  $|G_\omega|=60$.
\end{enumerate}

\end{corollary}
\begin{proof}
Let $G=\mathrm{Alt}(m)$, let $\Omega$ be the domain of $G$ and let $\omega\in \Omega$. When $G_\omega$ is intransitive on $\{1,\ldots,m\}$ the result follows from Lemma~\ref{Alt:type1}, when $G_\omega$ is imprimitive on $\{1,\ldots,m\}$ the result follows from Lemma~\ref{Alt:type2} and when $G_\omega$ is primitive on $\{1,\ldots,m\}$ the result follows from Lemma~\ref{Alt:type3}.
\end{proof}

We conclude this section by dealing with the sporadic simple groups.
\begin{lemma}\label{lemma:sporadic}Let $G$ be a sporadic simple group acting faithfully and transitively on a set of cardinality $n$. Then either $G$ has a semiregular subgroup of order $k$ with $k^{2k}\ge n$, or $G=M_{11}$ and $n=12$. Moreover, $G$ has a semiregular subgroup of order at least $4$ unless $G=M_{11}$ and $n=12$.
\end{lemma}
\begin{proof}
Let $\omega$ be a point in the domain $\Omega$ of $G$ and let $G_\omega$ be the stabilizer of $\omega$.

Let $p_1$ and $p_2$ be the two largest prime divisors of the order of $G$ with $p_2<p_1$. Using the order of the sporadic simple groups~\cite{ATLAS}, we have $p_2> 4$ and $p_1^{2p_1}>p_2^{2p_2}\ge |G|\ge |\Omega|=n$. In particular, if $G_\omega$ is relatively prime to $p_1$ or to  $p_2$, then the lemma follows immediately. Therefore, we may suppose that $p_1p_2$ divides $|G_\omega|$.

Suppose that $G$ is not the Monster. Let $M$ be a maximal subgroup of $G$ with $G_\omega\le M$. Using the information on~\cite{ATLAS}, we see that the order of $M$ is divisible by $p_1p_2$ only in one of the following cases:
\begin{itemize}
\item $G=Co_2$ and $M\cong M_{23}$,
\item $G=Co_3$ and $M\cong M_{23}$,
\item $G=McL$ and $M\cong M_{22}$,
\item $G=HS$ and $M\cong M_{22}$,
\item $G=M_{24}$ and $M\cong M_{23}$ or $M\cong\mathrm{PSL}_2(23)$,
\item $G=M_{23}$ and $M\cong 23:11$,
\item $G=M_{12}$ and $M\cong M_{11}$ or $M\cong\mathrm{PSL}_2(11)$,
\item $G=M_{11}$ and $M\cong\mathrm{PSL}_2(11)$.
\end{itemize}
Except when  $G=M_{11}$, for each of these cases, we have constructed with the help of a computer the permutation representation of $G$ on the cosets of $M$ and we have found a semiregular subgroup of order $k\ge 4$ with $k^{2k}\ge |G|$.

The group $G=M_{11}$ in its action on degree $12$ (on the right cosets of $\mathrm{PSL}_2(11)$) has no non-identity semiregular subgroups.\footnote{This fact was first proved by Giudici~\cite{Giudici}.} In particular, we obtain the exception listed in the statement of this lemma.

Finally, suppose $G$ is the Monster group $\mathbb M$.\footnote{It has been recently announced a complete classification of the maximal subgroups of the Monster, see~\cite{DLP}.} Here, $p_1=71$ and $p_2=59$. From~\cite[Section 3.6]{80}
and~\cite{81}, we see that the classification, up to isomorphism and up to conjugacy, of
the maximal subgroups of $G$ is complete except for a few open cases. In particular,
if $M$ is a maximal subgroup of $\mathbb M$, then either $M$ is in~\cite[Section 3.6]{80}, or the socle of
$M$ is $\mathrm{PSL}_2 (13)$ or $\mathrm{PSL}_2 (16)$. Therefore, from this list, we deduce that $G$ has no maximal subgroup whose order is divisible by $p_2p_1=59\cdot 71$.
\end{proof}

\section{Simple groups of Lie type}\label{sec:lie}

Given  a positive integer $x$, we let $\pi(x)$ denote the set of prime divisors of $x$. Moreover, given a finite group $G$, we let $\pi(G)$ denote the set of prime divisors of the order of $G$. For instance, when $G=\mathrm{Alt}(5)$, we have $\pi(G)=\{2,3,5\}$.

In this section we prove Theorem~\ref{thrm:main2} when $G$ is a simple groups of Lie type. Our main tool, besides the number theoretic results in Section~\ref{sec:number theory}, is a result of Liebeck, Praeger and Saxl~\cite[Theorem~4]{LPS}.\footnote{This result has already played an important role in other investigations on group actions on graphs. In particular, it is one of the ingredients for the proof of the Babai-Godsil conjecture on the asymptotic enumeration of Cayley digraphs~\cite{MS}.} We phrase it tailored to our current needs.
\begin{theorem}\label{thrm:LPS}
Let $T$ be a simple group of Lie type and let $M$ be a proper subgroup of $T$. Suppose that $|M|$ is divisible by each of the primes or prime powers indicated in the second or third column in \cite[Tables~$10.1$--$10.5$]{LPS}. Then the possibilities for $M$ are as given \cite[Tables~$10.1$--$10.5$]{LPS}.
\end{theorem}
For the rest of this section, we let $T$ be a simple group of Lie type acting primitively and faithfully on a set $\Omega$ and let $\omega\in \Omega$.  We apply Theorem~\ref{thrm:LPS} with $M=T_\omega$. In particular, our proof of Theorem~\ref{thrm:main2} for the action of $T$ on $\Omega$ splits into two major cases:
\begin{itemize}
\item[Case~1:]$|T_\omega|$ is not divisible by some prime power indicated in the second or third column in Tables~10.1--10.5 of~\cite{LPS},
\item[Case~2:]$|T_\omega|$ is divisible by each prime power indicated in the second and third column in Tables~10.1--10.5 of~\cite{LPS}.
\end{itemize}
In reading Tables~$10.1$--$10.5$ in~\cite{LPS}, we are only concerned in the case that $M$ is a maximal subgroup of $T$, because $M=T_\omega$ and $T$ is primitive on $\Omega$. Moreover, we are only interested in \textit{simple} Lie groups (see Proposition~\ref{prop:reduction}) and hence in these tables we are not concerned with groups that are not simple.

In Case~1, the number theoretic results in Section~\ref{sec:number theory} will show that $|T_\omega|$ is not divisible by a large prime, which yields a large semiregular subgroup for the action of $T$ on $\Omega$. Case~2 requires a detailed analysis on the pairs $(T,M)$ arising in Tables~10.1--10.5 of~\cite{LPS}.

Since we are aiming to determine the innately transitive groups with no semiregular subgroups of order at least $4$, both cases require special care. Therefore, in order to avoid cumbersome arguments, we deal with this special case with an ad-hoc argument in Section~\ref{lie4}.

Before embarking into these proofs, we make another observation again tailored to our needs. One remarkable application of~\cite[Theorem~4]{LPS} is a classification of all pairs $(T,M)$, where $T$ is a simple group of Lie type and  $M$ is a proper subgroup of $T$ with $\pi(T)=\pi(M)$. All of these pairs are reported in~\cite[Table~10.7]{LPS}. Here, we report  in Table~\ref{tabLE} lines 3, 4, 5 and 6  of~\cite[Table~10.7]{LPS}, because these play a special role in our arguments for dealing with Case~2.
\begin{table}[!ht]
\begin{tabular}{clcc}\hline
Line&$T$&$M$&Remarks\\\hline
1&$\mathrm{PSp}_{2m}(q)'$ &$\Omega_{2m}^-(q)\unlhd M$&$m$ and $q$ even\\
&&&${\bf N}_T(\Omega_{2m}^-(q))$  in the Aschbacher class $\mathcal{C}_8$\\\hline
2&$\mathrm{P}\Omega_{2m+1}(q)$ &$\Omega_{2m}^-(q)\unlhd M$&$m$ even and $q$ odd\\
&&&${\bf N}_T(\Omega_{2m}^-(q))$  in the Aschbacher class $\mathcal{C}_1$\\\hline
3&$\mathrm{P}\Omega_{2m}^+(q)$ &$\Omega_{2m-1}(q)\unlhd M$&$m$ even\\
&&&${\bf N}_T(\Omega_{2m-1}(q))$  in the Aschbacher class $\mathcal{C}_1$\\\hline
4&$\mathrm{PSp}_{4}(q)'$ &$\mathrm{PSp}_{2}(q^2)\unlhd M$&${\bf N}_T(\mathrm{PSp}_2(q^2))$  in the Aschbacher class $\mathcal{C}_3$\\\hline
\end{tabular}
\caption{Lines 3, 4, 5 and 6 of~\cite[Table~10.7]{LPS}}\label{tabLE}
\end{table}

\begin{lemma}\label{tableline1}Let $T=\mathrm{PSp}_{2m}(q)'$ be acting primitively and faithfully on a set $\Omega$ and let $\omega\in \Omega$.\footnote{Observe that $T$ is defined as the derived subgroup of $\mathrm{PSp}_{2m}(q)$, for including the case $(m,q)=(2,2)$, where $\mathrm{PSp}_4(2)\cong\mathrm{Sym}(6)$.} Assume that $T$ and  $T_\omega$ are as in the first line of Table~$\ref{tabLE}$. Then $T$ contains a semiregular subgroup of order at least $m\log_2q+1$.

Moreover, $T$ in its action on $\Omega$ has a semiregular subgroup of order at least $4$, unless $(m,q)=(2,2)$.
\end{lemma}
\begin{proof}
We deal with the case $(m,q)=(2,2)$ separately. Indeed, we have verified the veracity of the statement with \texttt{magma}~\cite{magma}. For the rest of the proof, we suppose $(m,q)\ne (2,2)$ and hence $T=\mathrm{PSp}_{2m}(q)$.

As $q$ is even, we have $T=\mathrm{Sp}_{2m}(q)$. Since $T_\omega$ is maximal in $T$, 
$$T_\omega={\bf N}_T(\Omega_{2m}^-(q))=\mathrm{SO}_{2m}^-(q)=\Omega_{2m}^-(q).2.$$
Let $q=2^f$, for some positive integer $f$. Fixing a suitable basis of the $\mathbb{F}_q$-vector space $V=\mathbb{F}_q^{2m}$, we may suppose that the symplectic form $\varphi$ preserved by $T$ has matrix
\[
\begin{pmatrix}
0&I\\
I&0
\end{pmatrix},
\]
where $I$ is the $m\times m$-identity matrix. 
 Suppose that $2^{fm}-1$ admits a primitive prime divisor and let $p$ be the largest such prime divisor. Let
$A\in\mathrm{GL}_{m}(q)$ be an element having order $p$ and let
$$g=
\begin{pmatrix}
A&0\\
0&(A^{-1})^T
\end{pmatrix}.
$$
An easy computation shows that $g$ preserves $\varphi$ and hence $g\in \mathrm{Sp}_{2m}(q)=T$. Every non-identity element of $X=\langle g\rangle$ fixes only two distinct non-trivial subspaces of $V$, namely, $V_1=\langle e_1,\ldots,e_m\rangle$ and $V_2=\langle e_{m+1},\ldots,e_{2m}\rangle$, where $e_1,\ldots,e_{2m}$ is the canonical basis of $V=\mathbb{F}_q^{2m}$.

Assume, by contradiction, that $X$ has a $T$-conjugate in $T_\omega=\mathrm{SO}_{2m}^-(q)$. Replacing $T_\omega$ by a suitable $T$-conjugate, we may suppose that $X\le T_\omega$. Let $Q$ be the quadratic form preserved by $T_\omega$. Since $X$ acts irreducibly on $V_1$ and on $V_2$ and since $X$ preserves $Q$, we deduce that either $V_i$ is totally singular for $Q$ or $V_i$ is non-degenerate for $Q$. Since $Q$ has Witt defect $1$ and since $\dim_{\mathbb{F}_q}(V_i)=m$, we deduce that $V_i$ is non-degenerate and hence the quadratic form $Q$ restricted to $V_i$ induces a non-degenerate quadratic form $Q_i$. As $X$ acts irreducibly on $V_i$ and as $X$ preserves $Q_i$, we deduce from~\cite{huppert} that $Q_i$ has Witt defect $1$.\footnote{From~\cite{huppert}, the group $\mathrm{SO}_{m}^+(q)$ does not contain elements acting irreducibly on the underlying vector space.} As $Q=Q_1\oplus Q_2$ and as $Q_1$, $Q_2$ have both Witt defect $1$, we deduce that $Q$ has Witt defect $0$, which is a contradiction. This contradiction has shown that no $T$-conjugate of $X$ lies in $T_\omega$ and hence $X$ acts semiregularly on $\Omega$. Since $|X|=p\ge fm+1=m\log_2q+1$, the first part of the lemma follows in this case.

Suppose that $2^{fm}-1$ does not admit a primitive prime divisor. From~\cite{zi}, this implies $(f,m)\in \{(1,2),(3,2),(1,6)\}$. We have computed with a computer the size of semiregular subgroups in these cases and in each case there is a semiregular subgroup of order at least $m\log_2 q+1$.

It remains to discuss the existence of semiregular subgroups of order at least $4$. When $4\le m\log_2q+1$, this follows from the first part of the lemma. If $m\log_2q+1\le 3$, then $(m,q)=(2,2)$.
\end{proof}

\begin{lemma}\label{tableline2}Let $T=\mathrm{P}\Omega_{2m+1}(q)$ be acting primitively and faithfully on a set $\Omega$ and let $\omega\in \Omega$. Assume that $T$ and  $T_\omega$ are as in the second line of Table~$\ref{tabLE}$. Then $T$ contains a semiregular subgroup of order at least $(q^{m/2}+1)/4$  when $q\equiv 3\pmod 4$ and $m/2$ is odd, and of order at least $(q^{m/2}+1)/2$ in all other cases.

Moreover, $T$ in its action on $\Omega$ has a semiregular subgroup of order at least $4$.
\end{lemma}
\begin{proof}
Clearly, $T=\Omega_{2m+1}(q)$, because $\Omega_{2m+1}(q)$ is centerless. Here $T_\omega$ is the stabilizer of a non-singular $1$-dimensional subspace $\langle v\rangle$ of $V=\mathbb{F}_q^{2m+1}$ such that the non-degenerate orthogonal form $Q$ for $T=\mathrm{P}\Omega_{2m+1}(q)=\Omega_{2m+1}(q)$ restricted to $\langle v\rangle^\perp$ has Witt defect $1$. Let $w\in V\setminus\{0\}$ such that the quadratic form $Q$ restricted to $\langle w\rangle^\perp$ has Witt defect $0$. Thus the orthogonal decomposition $V=\langle w\rangle\perp \langle w\rangle^\perp$ gives rise to an embedding of $\Omega_1(q)\times \Omega_{2m}^+(q)=\Omega_{2m}^+(q)$ in $\Omega_{2m+1}(q)=T$. The vector space $W=\langle w\rangle^\perp$ is endowed with the non-degenerate quadratic form $Q_{|W}$ having Witt defect zero and hence $W$ admits a direct sum decomposition
$$W=W_1\oplus W_2,$$
where $\dim_{\mathbb{F}_q}(W_i)=m$ and the quadratic form $Q_{|W}$ restricted to $W_i$ has Witt defect $1$. Using this orthogonal decomposition, we deduce the embedding $\Omega_m^-(q)\times\Omega_m^-(q)\le \Omega_{2m}^+(q)$. By~\cite{huppert}, $\Omega_m^-(q)$ contains a cyclic subgroup of order $(q^{m/2}+1)/2$ acting as a scalar in $\mathbb{F}_{q^m}$, when the $\mathbb{F}_q$-vector space $W_i\cong \mathbb{F}_q^m$ is identified with the additive group of the field $\mathbb{F}_{q^m}$. Let $x_i$ be a generator of this cyclic subgroup. Let $$\ell\in \left\{0,\ldots,\frac{q^{m/2}+1}{2}-1\right\}$$ be a divisor of $(q^{m/2}+1)/2$ and suppose that $x_i^\ell$ fixes a $1$-dimensional subspace of $W_i$. Then $x_i^\ell$ has $m$ eigenvalues in $\mathbb{F}_q$ and hence $$\frac{q^{m/2}+1}{2\ell}$$
divides $q-1$. Observe that
\[
\gcd((q^{m/2}+1)/2,q-1)=
\begin{cases}
2&\textrm{when }q\equiv 3\pmod 4\textrm{ and }m/2\textrm{ is odd},\\
1&\textrm{otherwise}.
\end{cases}
\]Moreover, when $q\equiv 3\pmod 4$ and $m/2$ is odd, $x_i^{(q^{m/2}+1)/4}$ is the scalar matrix $-1$. Therefore $\langle x_i\rangle/\{1,-1\}$ acts semiregularly on the $1$-dimensional subspaces of $W_i$ when $q\equiv 3\pmod 4$ and $m/2$ is odd, and $\langle x_i\rangle$ acts semiregularly on the $1$-dimensional subspaces of $W_i$ in all the remaining cases.

Let  $g=x_1\oplus x_2\in \Omega_{2m}^+(q)\le \Omega_{2m+1}(q)$ and let $X=\langle g\rangle$. Now, $X$ has order $(q^{m/2}+1)/4$ when $q\equiv 3\pmod 4$ and $m/2$ is odd, and $X$ has order $(q^{m/2}+1)/2$ in all other cases.
From the discussion above, we see that every non-identity element of $X$ fixes only the $1$-dimensional subspace $\langle w\rangle$ and hence it is a derangement for the action on $\Omega$.

It remains to discuss the existence of semiregular subgroups of order at least $4$. When $m\ge 4$, this follows from the first part of the proof. Suppose then $m=2$. Let $\varepsilon=2$ when $q\equiv 1\pmod 4$ and $\varepsilon=4$ when $q\equiv 3\pmod 4$. Now, $(q+1)/\varepsilon\ge 4$ only when $q\notin \{3,5,7,11\}$. Finally, when $m=2$ and  $q\in \{3,5,7,11\}$, we may use a regular unipotent element of $\Omega_5(q)$ to obtain a semiregular subgroup of order $9$ when $q=3$ and of order $q$ when $q\in \{5,7,11\}$.
\end{proof}

\begin{lemma}\label{tableline3}Let $T=\mathrm{P}\Omega_{2m}^+(q)$ be acting primitively and faithfully on a set $\Omega$ and let $\omega\in \Omega$. Assume that $T$ and  $T_\omega$ are as in the third line of Table~$\ref{tabLE}$. Then $T$ contains a semiregular subgroup of order at least $(q^{m/2}+1)/4$ when $q\equiv 3\pmod 4$  and $m/2$ is odd, and of order at least $(q^{m/2}+1)/\gcd(2,q-1)$ in all other cases.

Moreover, $T$ in its action on $\Omega$ has a semiregular subgroup of order at least $4$.
\end{lemma}
\begin{proof}
Here $T_\omega={\bf N}_{T}(\Omega_{2m-1}(q))$ and hence $T_\omega$ is the stabilizer of a non-singular $1$-dimensional subspace $V=\mathbb{F}_q^{2m}$. Observe that when $q$ is even, $\Omega_{2m-1}(q)=\mathrm{Sp}_{2m-2}(q)$.  Moreover, $m\ge 4$, because $\mathrm{P}\Omega_4^+(q)\cong\mathrm{PSL}_2(q)\times\mathrm{PSL}_2(q)$ is not simple.

The vector space $V$ is endowed with a non-degenerate quadratic form having Witt defect zero. Therefore, $V$ admits a direct sum decomposition
$$V=V_1\oplus V_2,$$
where $\dim_{\mathbb{F}_q}(V_i)=m$ and the quadratic form restricted to $V_i$ has Witt defect $1$. Using this orthogonal decomposition, we deduce the embedding $\Omega_m^-(q)\times\Omega_m^-(q)\le \Omega_{2m}^+(q)$. Now, by~\cite{huppert}, $\Omega_m^-(q)$ contains a cyclic subgroup of order $$\frac{q^{m/2}+1}{\gcd(2,q-1)}$$ acting as a scalar in $\mathbb{F}_{q^m}$, when the $\mathbb{F}_q$-vector space $V_i\cong \mathbb{F}_q^m$ is identified with $\mathbb{F}_{q^m}$. (The argument here is similar to the proof of Lemma~\ref{tableline2}.) Let $x_i$ be a generator of this cyclic subgroup. Let $\ell\in \{0,\ldots,(q^{m/2}+1)/\gcd(2,q-1)-1\}$ be a divisor of $(q^{m/2}+1)/\gcd(2,q-1)$ and suppose that $x_i^\ell$ fixes a $1$-dimensional subspace of $V_i$. Then $x_i^\ell$ has $m$ eigenvalues in $\mathbb{F}_q$ and hence $$\frac{q^{m/2}+1}{\ell\gcd(2,q-1)}$$
divides $q-1$. Observe that
\[
\gcd((q^{m/2}+1)/\gcd(2,q-1),q-1)=
\begin{cases}
2&\textrm{when }q\equiv 3\pmod 4\textrm{ and }m/2\textrm{ is odd},\\
1&\textrm{otherwise}.
\end{cases}
\]Moreover, when $q\equiv 3\pmod 4$ and $m/2$ is odd, $x_i^{(q^{m/2}+1)/4}$ is the scalar matrix $-1$. Therefore $\langle x_i\rangle/\{1,-1\}$ acts semiregularly on the $1$-dimensional subspaces of $V_i$ when $q\equiv 3\pmod 4$ and $m/2$ is odd, and $\langle x_i\rangle$ acts semiregularly on the $1$-dimensional subspaces of $V_i$ in all the remaining cases.

Let  $\tilde g=x_1\oplus x_2\in \Omega_{2m}^+(q)$, let $g$ be the projective image of $\tilde g$ in $\mathrm{P}\Omega_{2m}^+(q)$ and let $X=\langle g\rangle$. Now, $X$ has order $(q^{m/2}+1)/4$ when $q\equiv 3\pmod 4$ and $m/2$ is odd, and $X$ has order $(q^{m/2}+1)/\gcd(2,q-1)$ in all other cases.
From the discussion above, we see that every non-identity element of $X$ is a derangement for the action on the non-degenerate $1$-dimensional subspaces of $V=\mathbb{F}_q^{2m}$.

It remains to discuss the existence of semiregular subgroups of order at least $4$. As $m\ge 4$, this follows from the first part of the proof.
\end{proof}

\begin{lemma}\label{tableline4}Let $T=\mathrm{PSp}_{4}(q)'$ be acting primitively and faithfully on a set $\Omega$ and let $\omega\in \Omega$. Assume that $T$ and  $T_\omega$ are as in the fourth line of Table~$\ref{tabLE}$. Then $T$ contains a semiregular subgroup of order at least $q^2$ when $q$ is odd  and of order at least $\log_2q+1$ when $q$ is even.

Moreover, $T$ in its action on $\Omega$ has a semiregular subgroup of order at least $4$, except when $q=2$.
\end{lemma}
\begin{proof}
Without loss of generality, we may suppose that the alternating form defining $T=\mathrm{PSp}_4(q)$ is
\[
\begin{pmatrix}
0&0&1&0\\
0&0&0&1\\
-1&0&0&0\\
0&-1&0&0
\end{pmatrix}.
\]

Suppose that $q$ is odd. The  unipotent elements of $T_\omega$ lie in $\mathrm{PSp}_2(q^2)$ because $q$ is odd. Hence the non-identity unipotent elements of  $T_\omega$ have two Jordan blocks of size $2$, because $\mathrm{PSp}_2(q^2)$ preserves an extension field. Now, consider the subgroup $X$ of $T$ consisting of the matrices
\[
\begin{pmatrix}
1&0&a&b\\
0&1&0&0\\
0&0&1&0\\
0&0&0&1
\end{pmatrix},\, \forall a,b\in\mathbb{F}_q.
\]
The non-identity elements of $X$ have two Jordan blocks of size 1 and one Jordan block of size 2. Therefore, except for the identity, none of the elements of $X$ is $T$-conjugate to an element of $T_\omega$. Therefore $X$ is a semiregular subgroup of order $q^2$.

Suppose that $q$ is even. Write $q=2^f$, for some positive integer $f$. Thus $T=\mathrm{Sp}_4(q)$ and $\mathrm{Sp}_2(q^2)\unlhd T_\omega= \mathrm{Sp}_2(q^2):2$. The elements of odd order of $T_\omega$ lie in $\mathrm{Sp}_2(q^2)$ and hence the elements of odd order of $T_\omega$ have either zero or two (with multiplicity two) eigenvalues in $\mathbb{F}_q$. Assume that $2^f-1$ is divisible by a primitive prime divisor $p$ with $p\ne 3$. From~\cite{zi}, this implies $f\notin\{2,6\}$. Clearly, $p\ge f+1$. Let $\lambda\in \mathbb{F}_q^\ast$ having order $p$.  Now, consider the subgroup $X$ of $T$ consisting of the matrices
\[
\begin{pmatrix}
a&0&0&0\\
0&a^2&0&0\\
0&0&a^{-1}&0\\
0&0&0&a^{-2}
\end{pmatrix}, \,\forall a\in\langle\lambda\rangle\subseteq\mathbb{F}_q^\ast.
\]
As $p\ne 3$, the non-identity elements of $X$ have four distinct eigenvalues. Therefore, except for the identity, none of the elements of $X$ is $T$-conjugate to an element of $T_\omega$. Therefore $X$ is a semiregular subgroup of order $p\ge f+1$. We have verified with a computer that, when $f\in \{2,6\}$, the group $T=\mathrm{Sp}_4(q)$ admits a semiregular subgroup of order at least $f$.

It remains to discuss the existence of semiregular subgroups of order at least $4$. When $q$ is odd this is clear because $q^2\ge 4$. When $q$ is even and $\log_2q+1\ge 4$, this is also clear. When $q$ is even and $\log_2q+1\le 3$, we have $q\in \{2,4\}$. We have verified with a computer that when $q=4$, $T$ in its action on $\Omega$ has a semiregular subgroup of order $15$.
\end{proof}

\subsection{Semiregular subgroups of large order}\label{sectioncase1}

Recall that $T$ is a simple group of Lie type acting primitively and faithfully on a set $\Omega$ and  $\omega\in \Omega$. Moreover, we apply Theorem~\ref{thrm:LPS} with $M=T_\omega$.

\begin{lemma}\label{lemma:-11}
There exists a function $f:\mathbb{N}\to \mathbb{N}$ such that, if $T$ is a simple group of Lie type acting primitively and faithfully on a set $\Omega$, then either $T$ in its action on $\Omega$ has a semiregular subgroup of order at least $k$ or $|\Omega|\le f(k)$.
\end{lemma}
\begin{proof}
The proof follows easy with a careful inspection on Tables~10.1--10.5 in~\cite{LPS}. Here we use the notation from~\cite{LPS} and we give details only for a few cases, all other cases are dealt with similarly.

\smallskip

Suppose $T=\mathrm{PSL}_n(q)$. Assume $n\ge 6$ even and let $\Pi=\{q_n,q_{n-1}\}$. From~\cite[Table~10.1]{LPS}, we deduce that, with the exception of $T=\mathrm{PSL}_6(2)$, there exists no proper subgroup $M$ of $T$ with $|M|$ divisible by each prime in $\Pi$. Since we may exclude $\mathrm{PSL}_6(2)$ from this asymptotic result, we deduce that $|T_\omega|$ is not divisible by some prime $p$ in $\Pi$. In particular, $T$ has a semiregular subgroup of order at least $p$. By using Lemma~\ref{siegeltheorem}, $p$ tends to infinity as $|T|$ tends to infinity. Assume $n\ge 5$ is odd and let $\Pi=\{q_n,q_{n-1},q_{n-2}\}$. From~\cite[Table~10.1]{LPS}, we deduce that, with the exception of $T=\mathrm{PSL}_7(2)$, there exists no proper subgroup $M$ of $T$ with $|M|$ divisible by each prime in $\Pi$. In particular, we may argue as above for dealing with this case. The argument for $\mathrm{PSL}_2(q)$, $\mathrm{PSL}_3(q)$ and $\mathrm{PSL}_4(q)$ is entirely similar, using Table~10.3 in~\cite{LPS}.

\smallskip

Suppose $T=\mathrm{PSp}_{2m}(q)'$. Assume $m\ge 3$ odd and let $\Pi=\{q_{2m-2},q_{2m-2},q_{2m}\}$. From~\cite[Table~10.1]{LPS}, we deduce that, with the exception of $T=\mathrm{PSp}_6(2)$, there exists no proper subgroup $M$ of $T$ with $|M|$ divisible by each prime in $\Pi$. Since we may exclude $\mathrm{PSp}_6(2)$ from this asymptotic result, we deduce that $|T_\omega|$ is not divisible by some prime $p$ in $\Pi$. In particular, $T$ has a semiregular subgroup of order at least $p$. By using Lemma~\ref{siegeltheorem}, $p$ tends to infinity as $|T|$ tends to infinity. Assume $m\ge 4$ is even and let $\Pi=\{q_{2m},q_{2m-2},q_{2m-4}\}$. From~\cite[Table~10.1]{LPS}, we deduce that, with the exception of $T=\mathrm{PSp}_8(2)$, the only maximal with $|M|$ divisible by each prime in $\Pi$ satisfies $M={\bf N}_T(\Omega_{2m}^-(q))$. In particular, if $T_\omega$ is not divisible by some prime in $\Pi$, then we deduce that $T$ has semiregular subgroups of large order from Lemma~\ref{siegeltheorem}. Therefore, we just need to consider the action of $T=\mathrm{PSp}_{2m}(q)$ on the right cosets of $T_\omega={\bf N}_T(\Omega_{2m}^-(q))$. Lemma~\ref{tableline1} deals exactly with this action and indeed, it shows that $T$ contain semiregular subgroups having order that tends to infinity as $|T|$ tends to infinity. The argument for $\mathrm{PSp}_4(q)$ is similar and uses Table~10.3 in~\cite{LPS} and Lemma~\ref{tableline4}.

\smallskip

Suppose $T=\mathrm{P}\Omega_{2m+1}(q))$. Assume $m\ge 3$ odd and let $\Pi=\{q_{2m-2},q_{2m-2},q_{2m}\}$. From~\cite[Table~10.1]{LPS}, we deduce that there exists no proper subgroup $M$ of $T$ with $|M|$ divisible by each prime in $\Pi$.  We deduce then that $|T_\omega|$ is not divisible by some prime $p$ in $\Pi$. In particular, $T$ has a semiregular subgroup of order at least $p$. By using Lemma~\ref{siegeltheorem}, $p$ tends to infinity as $|T|$ tends to infinity. Assume $m\ge 4$ is even and let $\Pi=\{q_{2m},q_{2m-2},q_{2m-4}\}$. From~\cite[Table~10.1]{LPS}, we deduce that the only maximal with $|M|$ divisible by each prime in $\Pi$ satisfies $M={\bf N}_T(\Omega_{2m}^-(q))$. In particular, if $T_\omega$ is not divisible by some prime in $\Pi$, then we deduce that $T$ has semiregular subgroups of large order from Lemma~\ref{siegeltheorem}. Therefore, we just need to consider the action of $T=\mathrm{P}\Omega_{2m+1}(q)$ on the right cosets of $T_\omega={\bf N}_T(\Omega_{2m}^-(q))$. Lemma~\ref{tableline2} deals exactly with this action and indeed, it shows that $T$ contain semiregular subgroups having order that tends to infinity as $|T|$ tends to infinity.
\smallskip

The argument for all other Lie type groups is similar and it is omitted.
\end{proof}

Observe that the result in our Section~\ref{sectioncase1} can be seen as an asymptotic improvement of Corollary~6 in~\cite{LPS}.

\subsection{Semiregular subgroups of order at least four}\label{lie4}

We report here~\cite[Corollary~7]{LPS}. Again, we only state it for our current needs.
\begin{lemma}\label{Lie4lemma}
Let $T$ be a simple group of Lie type and assume $$T\ne \mathrm{PSL}_2(8), \mathrm{PSL}_3(3), \mathrm{PSU}_3(3), \mathrm{PSp}_4(8), \hbox{ or }\mathrm{PSL}_2(p)$$ with $p$ a Mersenne prime. Then there is a collection $\Pi$ of prime numbers of $|T|$, such that for $M< T$, if $\Pi\subseteq \pi(M)$, then $\pi(T)=\pi(M)$ and $M$ is given in Table~$10.7$ in~\cite{LPS}. Moreover, every prime in $\Pi$ is at least $5$, except in the following cases.
\begin{tabular}{lc}\hline
$T$&$\Pi$\\\hline
$\mathrm{PSL}_2(p)$, $p$ \textrm{ prime, }$p=2^{a}3^b-1$, $b>0$&$\{3,p\}$\\
$\mathrm{PSU}_4(2)$&$\{3,5\}$\\
$\mathrm{PSU}_5(2)$&$\{3,5,11\}$\\\hline
\end{tabular}
\end{lemma}

\begin{lemma}\label{Lie4}
Let $T$ be a simple group of Lie type acting primitively and faithfully on a set $\Omega$. Then either $T$ contains a semiregular subgroup of order at least $4$, or one of the following holds
\begin{enumerate}
\item\label{excep1} $T\cong\mathrm{PSL}_2(4)\cong\mathrm{PSL}_2(5)$ and $|\Omega|=6$,
\item\label{excep2} $T\cong\mathrm{PSL}_2(9)$ and $|\Omega|=6$,
\item\label{excep3} $T\cong\mathrm{PSU}_3(3)$ and $|\Omega|=36$.
\end{enumerate}
\end{lemma}
\begin{proof}
Let $\omega\in \Omega$ and set $M=T_\omega$.

Suppose first that $T$ is isomorphic to
$$\mathrm{PSL}_2(8), \mathrm{PSL}_3(3), \mathrm{PSU}_3(3), \mathrm{PSp}_4(8), \mathrm{PSU}_4(2), \mathrm{PSU}_5(2)$$
or to a group in~\cite[Table~10.7]{LPS}, except for the first seven rows.
In this case, the proof follows with a computer computation with the algebra system \texttt{magma}. Indeed, in each of these cases, the group $T$ is small and the result can be verified with the auxiliary help of a computer by constructing with a case-by-case analysis the primitive permutation representations under consideration and by checking the existence of semiregular subgroups of order at least $4$. Therefore, for the rest of the proof, we may suppose that $T$ is not isomorphic to any of these groups.\footnote{Observe that the exceptional case~\eqref{excep3} arises when analyzing these small groups.}

Suppose $T=\mathrm{PSL}_2(p)$, with $p$  prime. If $\gcd(|M|,p)=1$, then a Sylow $p$-subgroup of $T$ acts semiregularly on $\Omega$. Therefore, as $p\ge 4$ (because $p$ is prime), we deduce that $T$ has  a semiregular subgroup of order at least $4$. If $p$ divides $|M|$, then $M$ is a Borel subgroup of $T$ and the action of $T$ on $\Omega$ is permutation equivalent to the action of $T$ on the points of the projective line. Therefore, $T$ has a semiregular element of order $(p+1)/2$. Now, $(p+1)/2\ge 4$, except when $p=5$: this is the exception in~\eqref{excep1}. Therefore, for the rest of the proof, we may suppose that $T$ is not isomorphic to $\mathrm{PSL}_2(p)$, with $p$  prime.

We are now in the position to use Lemma~\ref{Lie4lemma}. There exists a set $\Pi$ of three prime numbers, each at least 5, with the property that either $\Pi\nsubseteq \Pi(M)$, or $\Pi\subseteq \pi(M)$ and $(T,M)$ is one of the pairs in the first seven rows of Table~10.7 in~\cite{LPS}. In the first case we are done, because $T$ has a semiregular subgroup of order at least $5\ge 4$. If $(T,M)$ are as in lines 3, 4, 5 and 6 of~\cite[Table~10.7]{LPS}, then the result follows from Lemmas~\ref{tableline1}--\ref{tableline4}. Therefore it remains to consider the lines 1, 2 and 7 of~\cite[Table~10.7]{LPS}. In lines 1 and 2, the group $T$ is alternating and hence only $\mathrm{Alt}(5),\mathrm{Alt}(6),\mathrm{Alt}(8)$ are of interest here. A computation yields that only the examples in~\eqref{excep1} and~\eqref{excep2} have no semiregular subgroup of order at least $4$. Finally, in line 7, we have $T=\mathrm{PSL}_2(p)$ with $p$ prime, which we have dealt with above.
\end{proof}

\end{document}